\documentclass[11pt, reqno]{amsart}
\usepackage{amssymb,latexsym,amsfonts}
\usepackage{amsthm}
\usepackage{amsmath,calligra,mathrsfs}
\usepackage{color}
\usepackage[OT2, T1]{fontenc}
\usepackage{url}
\usepackage{array}
\usepackage{hyperref}
\usepackage{paralist}
\usepackage{mathabx}			% for \widec
\usepackage{amscd}
\usepackage[all,cmtip]{xy}			% for commutative diagrams
\usepackage{sseq}				% for spectral sequences
\usepackage{verbatim}
\usepackage{parskip}
\usepackage{microtype}
\usepackage[in]{fullpage}
\usepackage{mathrsfs} 			% for \mathscr (script letters)
\usepackage{cleveref}
\usepackage{enumitem}
\usepackage{moreenum}			% for enumeration via Greek letters
\usepackage[alphabetic,lite]{amsrefs} 	% for bibliography
\usepackage{colonequals}	
\usepackage{fontenc}
\usepackage{amsmath}
\usepackage{geometry}
\usepackage{graphicx}
\usepackage{stmaryrd}
\usepackage[all]{xy}
\usepackage[french]{babel}
\usepackage[latin1]{inputenc}
\usepackage[T1]{fontenc}
\usepackage[foot]{amsaddr}
\usepackage{subfiles}
\usepackage{ragged2e}
\usepackage[bbgreekl]{mathbbol}     % This package is in order to have the prism symbol
\DeclareSymbolFontAlphabet{\mathbb}{AMSb}   % This is in order to common blackboard letters in the usual style (else the prism package messes them up)

\DeclareMathOperator{\et}{\text{ét}}

\DeclareMathOperator{\Ass}{Ass}
\newcommand{\clbul}{\mathcal{L}^{\bullet}X}

\DeclareMathOperator{\RHom}{RHom}
\DeclareMathOperator{\cL}{\mathcal{L}}

\DeclareMathOperator{\Spec}{Spec}
\DeclareMathOperator{\Fitt}{Fitt}

\DeclareMathOperator{\nil}{nil}

\DeclareMathOperator{\Id}{Id}
\DeclareMathOperator{\Hom}{Hom}

\DeclareMathOperator{\Ima}{Im}

\DeclareMathOperator{\Inf}{Inf}

\DeclareMathOperator{\Sh}{Sh}

\DeclareMathOperator{\Rhom}{R\mathscr{H}\text{\kern -3pt {\calligra\large om}}\,}
\DeclareMathOperator{\EExt}{\mathscr{E}\text{\kern -3pt {\calligra\large xt}}\,}

\DeclareMathOperator{\M}{M}
\DeclareMathOperator{\sinf}{<}

\DeclareMathOperator{\Ker}{Ker}

\DeclareMathOperator{\car}{car}

\DeclareMathOperator{\Def}{Def}

\DeclareMathOperator{\degtr}{degtr}

\DeclareMathOperator{\X}{X_{\text{proét}}}

\DeclareMathOperator{\Tet}{T_{\text{ét}}}

\DeclareMathOperator{\colim}{colim}

\DeclareMathOperator{\Set}{S_{\text{ét}}}

\DeclareMathOperator{\pr}{pr}

\DeclareMathOperator{\gr}{gr}

\DeclareMathOperator{\Jac}{Jac}
\newcommand{\cll}{\mathcal{L} X}
\DeclareMathOperator{\clp}{\mathcal{L}^{+} }
\DeclareMathOperator{\cLp}{\mathcal{L}^{+} }
\DeclareMathOperator{\clxD}{\mathcal{L}_{d} X}

\newcommand{\clui}{\mathcal{L} U_{i}}
\DeclareMathOperator{\clx}{\mathcal{L} X^{\leq d}}
\DeclareMathOperator{\clrx}{\mathcal{L}_{r}X}
\DeclareMathOperator{\clxd}{\mathcal{L} X^{=d}}

\usepackage{aliascnt}
\newaliascnt{numberingbase}{subsubsection}
\numberwithin{equation}{numberingbase}

\newtheoremstyle{thms}{0.7em}{0pt}{\itshape}{}{\bfseries}{.}{ }{}
\theoremstyle{thms}
\newtheorem{conj}[numberingbase]{Conjecture}
\newtheorem{cor}[numberingbase]{Corollaire}
\newtheorem{lemma}[numberingbase]{Lemme}
\newtheorem{lem}[numberingbase]{Lemme}
\newtheorem{prop}[numberingbase]{Proposition}
\newtheorem{Q}[numberingbase]{Question}
\newtheorem{thm}[numberingbase]{Théorème}
\newtheorem{theorem}[numberingbase]{Théorème}

\newtheoremstyle{claims}{0.7em}{0pt}{}{}{\itshape}{.}{ }{}
\theoremstyle{claims}
\newtheorem{claim}[equation]{Claim}

\newtheoremstyle{defs}{0.7em}{0pt}{}{}{\bfseries}{.}{ }{}
\theoremstyle{defs}
\newtheorem{defi}[numberingbase]{Définition}

\newtheorem{exa}[numberingbase]{Exemple}
\newtheorem*{exas}{Exemples}
\newtheorem{rmq}[numberingbase]{Remarque}
\newtheorem*{rmqs}{Remarques}

\Crefname{claim}{Claim}{Claims}
\Crefname{conj}{Conjecture}{Conjectures}
\Crefname{cor}{Corollaire}{Corollaires}
\Crefname{defn}{Définition}{Définitions}
\Crefname{eg}{Exemple}{Exemples}
\Crefname{prop}{Proposition}{Propositions} 
\Crefname{Q}{Question}{Questions}
\Crefname{rem}{Remarques}{Remarques}
\Crefname{theorem}{Théorème}{Théorèmes}
\Crefname{thm}{Théorème}{Théorèmes}
\Crefname{variant}{Variant}{Variants}

\theoremstyle{thms}
\newtheorem{thm-tweak}[subsection]{Théorème}
\Crefname{thm-tweak}{théorème}{théorèmes}
\newtheorem{lemma-tweak}[subsection]{Lemme}
\Crefname{lemma-tweak}{lemme}{Lemmes}
\newtheorem{cor-tweak}[subsection]{Corollaire}
\Crefname{cor-tweak}{corollaire}{corollaires}
\newtheorem{prop-tweak}[subsection]{Proposition}
\Crefname{prop-tweak}{proposition}{propositions} 
\newtheorem{conj-tweak}[subsection]{Conjecture}
\Crefname{conj-tweak}{Conjecture}{Conjectures} 

\theoremstyle{defs}
\newtheorem{defn-tweak}[subsection]{Définition}
\Crefname{defn-tweak}{définition}{définitions}
\newtheorem{eg-tweak}[subsection]{Exemple}
\Crefname{eg-tweak}{exemple}{exemples}
\newtheorem*{rmqs-tweak}{Remarques}
\newtheorem{rmq-tweak}[subsection]{Remarque}
\Crefname{rmq-tweak}{remarque}{remarques}

\newtheoremstyle{subsection-tweak}
   {11pt}
   {3pt}%
   {}
   {}%
   {\bfseries}
   {}%
   {.5em}
   {\thmnumber{\@{#1}{}\@{#2}.}%
    \thmnote{~{\bfseries#3.}}}

\theoremstyle{subsection-tweak}
\newtheorem{pp}[numberingbase]{}
\newcommand{\bpp}{\begin{pp}}
\newcommand{\epp}{\end{pp}}

\theoremstyle{subsection-tweak}
\newtheorem{pp-tweak}[subsection]{}

\renewcommand{\b}{\textbf}

\newcommand{\brems}{\begin{rmqs} \hfill \begin{enumerate}[label=\b{\thenumberingbase.},ref=\thenumberingbase]}
\newcommand{\remi}{\addtocounter{numberingbase}{1} \item}
\newcommand{\erems}{\end{enumerate} \end{rmqs}}
\newcommand{\bexs}{\begin{exas} \hfill \begin{enumerate}[label=\b{\thenumberingbase.},ref=\thenumberingbase]}

\newcommand{\eexs}{\end{enumerate} \end{exas}}
\newcommand{\m}{\item}
\newcommand{\bsm}{\begin{smallmatrix}}
\newcommand{\esm}{\end{smallmatrix}}
\newcommand{\blem}{\begin{lemma}}
\newcommand{\elem}{\end{lemma}}

\newcommand{\bconj}{\begin{conj}}
\newcommand{\econj}{\end{conj}}
\newcommand{\bprob}{\begin{Problem}}
\newcommand{\eprob}{\end{Problem}}

\newcommand{\bq}{\begin{Q}}
\newcommand{\eq}{\end{Q}}
\newcommand{\benum}{\begin{enumerate}[label={{\upshape(\alph*)}}]}
\newcommand{\benuma}{\begin{enumerate}[label={{\upshape(\arabic*)}}]}
\newcommand{\benumr}{\begin{enumerate}[label={{\upshape(\roman*)}}]}
\newcommand{\eenum}{\end{enumerate}}
\newcommand{\bitem}{\begin{itemize}}
\newcommand{\eitem}{\end{itemize}}
\newcommand{\bc}{\begin{comment}}
\newcommand{\ec}{\end{comment}}
\newcommand{\bd}{\begin{defn}}
\newcommand{\ed}{\end{defn}}

\newcommand{\bex}{\begin{exa}}
\newcommand{\eex}{\end{exa}}

\newcommand{\bcl}{\begin{claim}}
\newcommand{\ecl}{\end{claim}}

\newcommand{\q}{\quad}

\newcommand{\ba}{\begin{aligned}}
\newcommand{\ea}{\end{aligned}}
\newcommand{\be}{\begin{equation}}
\newcommand{\ee}{\end{equation}}
\newcommand{\bpf}{\begin{proof}}
\newcommand{\epf}{\end{proof}}
\newcommand{\bthm}{\begin{thm}}
\newcommand{\ethm}{\end{thm}}
\newcommand{\bprop}{\begin{prop}}
\newcommand{\eprop}{\end{prop}}
\newcommand{\bcor}{\begin{cor}}
\newcommand{\ecor}{\end{cor}}
\newcommand{\brem}{\begin{rmq}}
\newcommand{\erem}{\end{rmq}}

\makeatletter
\def\@tocline#1#2#3#4#5#6#7{
    \begingroup 
    \@ifempty{#4}{}{}

    \parindent\z@ \leftskip#3\relax \advance\leftskip\@tempdima\relax
    #5\hskip-\@tempdima
      \ifcase #1
       \or\or \hskip 2em \or \hskip 1em \else \hskip 3em \fi%
      #6\nobreak\relax
    \dotfill\hbox to\@pnumwidth{\@tocpagenum{#7}}\par
    \nobreak
    \endgroup
 }
 \def\l@section{\@tocline{1}{0pt}{1pc}{}{}}

\renewcommand{\tocsection}[3]{%
  \indentlabel{\@ifnotempty{#2}{\makebox[1.3em][l]{%
    \ignorespaces#1 \bfseries{#2}.\hfill}}}\bfseries{#3}
    \vspace{1.5pt}}

\renewcommand{\tocsubsection}[3]{%
  \indentlabel{\@ifnotempty{#2}{\hspace*{-0.5em}\makebox[2.1em][l]{%
    \ignorespaces#1#2.\hfill}}}#3
    \vspace{1.5pt}}
   
\makeatother 

\makeatletter % This is for proper formatting of appendices in the table of contents
\newcommand\appendix@section[1]{%
  \refstepcounter{section}%
  \orig@section*{Appendix \@Alph\c@section. #1}%
}
\let\orig@section\section
\g@addto@macro\appendix{\let\section\appendix@section}
\makeatother

\newcommand{\wA}{\widehat{A}}
\newcommand{\wB}{\widehat{B}}

\newcommand{\wT}{\widehat{T}}

\newcommand{\cQ}{\mathcal{Q}}

\newcommand{\clrmi}{\mathcal{L}_{r}M_{i}}
\newcommand{\clmi}{\mathcal{L}M_{i}}

\newcommand{\cO}{\mathcal{O}}

\newcommand{\kp}{\mathfrak{p}}
\newcommand{\kq}{\mathfrak{q}}
\newcommand{\g}{\gamma}
\newcommand{\la}{\lambda}
\newcommand{\km}{\mathfrak{m}}

\newcommand{\ra}{\rightarrow}

\newcommand{\cI}{\mathcal{I}}

\newcommand{\bZ}{\mathbb{Z}}

\newcommand{\bql}{\overline{\mathbb{Q}}_{\ell}}

\newcommand{\ab}{\mathbb{A}}
\newcommand{\ev}{ev}

\newcommand{\bt}{\textbf{t}}

\newcommand{\NN}{\mathbb{N}}
\newcommand{\al}{\alpha}
\newcommand{\bG}{\mathbb{G}}

\newcommand{\hra}{\hookrightarrow}
\title{Cohomologie étale des espaces d'arcs }

\author{Alexis Bouthier}

\begin{document}
\maketitle

\begin{center}
\textbf{Résumé:}
\end{center}
 Dans un premier temps, on établit un théorème de structure sur les espaces d'arcs $\cL X$ où  $X$ est un $k$-schéma de type fini, $k$ un corps. Plus précisément, on montre que localement pour la topologie pro-lisse, l'espace d'arcs est un produit de la forme $T\times\ab^{\NN}$, où $T$ est un schéma non-noethérien qui recolle les complétés formels en les différents arcs, stratifié par des $k$-schémas de type fini. Cet énoncé globalise le théorème de Drinfeld-Grinberg-Kazhdan et montre une conjecture de Koll\`{a}r et Nemethi.
On utilise ensuite ce théorème de structure pour construire une catégorie dérivée de faisceaux constructibles, stable par les six opérations et on montre un énoncé de finitude pour la cohomologie. 
\bigskip

\begin{center}
\textbf{Abstract:}
\end{center}
 We establish a structure theorem on the arc space $\cL X$ of a $k$-scheme of finite type. More precisely, we show that the arc space is locally for the pro-smooth toplogy of the form $T\times\ab^{\NN}$ where $T$ is a non-noetherian scheme that glues the different completions, stratified by $k$-schemes of finite type. This statement globalizes Drinfeld-Grinberg-Kazhdan's theorem and proves a conjecture of Koll\`{a}r and Nemethi. Then, we use this theorem to construct a derived category of constructible sheaves, stable by six operations and we show a finiteness result for the cohomology. 

\hypersetup{
    linktoc=page,     %set to all if you want both sections and subsections linked
}
\renewcommand*\contentsname{}
\q\\
\tableofcontents

\section*{Introduction}
\subsection{Énoncés principaux}
Soit $k$ algébriquement clos. Soit $f:S'\rightarrow S$, un morphisme de $k$-schémas, il est dit pro-lisse, s'il s'écrit comme une limite projective filtrante de $S$-schémas lisses avec des morphismes de transition affines, arbitraires. Soit $X$ un $k$-schéma de type fini, $\cL X$ son espace d'arcs, on dispose d'une suite croissante d'ouverts $\cL X^{\leq n}$ où l'on borne la valuation de la singularité par $n$.
Avant d'énoncer les principaux théorèmes, on commence par rappeler l'énoncé de Drinfeld-Grinberg-Kazhdan \cite{Dr}, \cite{GK}.
Supposons $\dim X\geq 1$, sans composante isolée, soit $\g(t)\in(\clx)(k)$, alors le théorème de Drinfeld-Grinberg-Kazhdan (DGK) montre qu'il existe un $k$-schéma de type fini $Y$ et un $k$-point $y$ de $Y$ tel que l'on a un isomorphisme de voisinages formels:
\begin{center}
$\cll_{\g(t)}^{\wedge}\simeq (Y\times\ab^{\NN})^{\wedge}_{(y,0)}$.
\end{center}
où $\ab^{\NN}:=\Spec(k[x_{1},x_{2},\dots])$.
Le but est  d'obtenir cet énoncé dans un voisinage, pour une topologie appropriée, de l'arc $\g(t)$.
On renvoie au théorème \ref{fonda} pour un énoncé plus précis.
\begin{theorem}\label{thm1}
Soit $d\in\NN$, il existe un $k$-schéma $Z$ et un morphisme :
\begin{center}
$z: Z\rightarrow\clx$
\end{center}
sujet aux propriétés suivantes:
\begin{enumerate}
	\item 
	$z$ est pro-lisse, affine, surjectif et induit un isomorphisme sur les voisinages formels en tout $k$-point de $Z$.
	\item
	Le schéma $Z$ admet une immersion fermée nilpotente $Z\hookrightarrow T\times\ab^{\NN}$, qui induit un isomorphisme sur les voisinages formels de tout $k$-point et où $T$ est un schéma de type $(S)$ (cf. section \ref{(S)}).	
\end{enumerate}
\end{theorem}
Dans le cas analytique, l'existence de tels atlas avait été conjecturée par Koll\`{a}r-Nemethi \cite[Conj.73]{KN}.
La principale surprise est que l'on ne peut pas étendre la décomposition formelle de DGK en restant dans le monde des $k$-schémas de type fini, cela nécessite l'introduction d'une nouvelle classe de schémas, les schémas de type $(S)$. Ce sont des schémas non-noethériens, de dimension de Krull finie, qui admettent une stratification finie constructible par des $k$-schémas de type fini et dont tous les complétés des anneaux locaux aux points fermés sont noethériens. L'exemple de base est le schéma $S_1$ dont l'espace topologique s'identifie au constructible $\ab^{2}-\ab^1\cup\{0\}$.
L'autre ingrédient clé de cet énoncé de globalisation est l'utilisation de morphismes de type \og Weierstrass\fg~ qui sont pro-lisses et permettent après changement de base d'utiliser les théorèmes de division et de factorisation de type Weierstrass qui sont la clé de voûte de la décomposition formelle.
\medskip

Une fois que l'on a établi un énoncé géométrique, on s'intéresse à la théorie des faisceaux dessus, plus précisément, la partie affine $\ab^{\NN}$ étant facile à contrôler, le point clé est donc de montrer des énoncés de finitue pour les schémas de type $(S)$. 
Cet énoncé géométrique permet de définir une théorie des six foncteurs. 
On appelle $\cL$-schémas, tout $k$-schéma de présentation finie sur $\cL X^{\leq d}$ pour $d\in\NN$ et $X$ un $k$-schéma de type fini.
Notre deuxième énoncé est alors le suivant:
\begin{theorem}\label{thm2}
Soit $n\in\NN$ premier à la caractéristique de $k$, soit $f:S'\rightarrow S$ un morphisme séparé de présentation finie entre $\cL$-schémas, alors les foncteurs $(f_*,f^*,f_!,f^!,\RHom)$ préservent la catégorie des constructibles $D^b_{c}(S,\bZ/n)$.
\end{theorem}
On commence d'abord par obtenir la finitude pour des schémas de type $(S)$, cela s'obtient d'une part à l'aide d'un énoncé de dévissage des schémas de type $(S)$ (\ref{ouv}) et d'autre part à l'aide d'énoncés de changement de base pro-lisse et comparaison entre la cohomologie sur l'hensélisé et son complété dû à Gabber. On utilise ensuite les théorèmes de finitude de Gabber pour les schémas noethériens excellents, une fois que l'on s'est ramené à la complétion. Une fois ceci fait, ajouter la partie $\ab^{\NN}$ est aisé.
Le dernier énoncé est un énoncé de finitude que l'on peut voir comme une géométrisation de l'intégration motivique, ainsi que le considère Drinfeld dans \cite[sect. 6.8]{Dr}:
\begin{theorem}\label{thm3}
Soit $n\neq\car(k)$, $d\in\NN$, alors le complexe $R\Gamma(\cL X^{\leq d},\bZ/n)$ est constructible.
\end{theorem}
En vertu des triangles distingués relativement à une stratification et du théorèmes de finitude \ref{thm2}, il est suffisant d'établir cet énoncé le long d'une stratification constructible de $\cL X^{\leq d}$. On utilise alors le théorème \ref{fonda2} pour obtenir une stratification de $\cL X^{\leq d}$ par des schémas de la forme $H\times\ab^{\NN}$ où $H$ est un $k$-schéma de type fini et là on est ramené à l'énoncé de Drinfeld.

\subsection{Lien avec la théorie géométrique des représentations}
Ce travail tire sa source du souhait d'obtenir une bonne catégorie des faisceaux pervers sur les espaces d'arcs. Cette hypothétique catégorie de faisceaux pervers n'a cessé de faire son apparition comme fil directeur de nombreux travaux, d'abord en théorie de la représentation et en Langlands géométrique (\cite{FMI}, \cite{BFGM}, \cite{FGV}) puis plus récemment dans l'étude des fonctions $L$ (\cite{Sak}, \cite{Ngo}).
\medskip

En l'absence d'une telle catégorie, la méthode dont on disposait jusqu'alors était de remplacer ces objets de nature locale de dimension infinie, par des espaces de modules globaux sur lesquels on pouvait par exemple calculer des complexes d'intersections. On retrouve ce type de techniques, entre autres, dans les travaux de Finkelberg-Mirkovic \cite{FMI}, Braverman-Gaitsgory-Finkelberg-Mirkovic \cite{BFGM} et Bouthier-Ngô-Sakellaridis \cite{BNS}.
Pour s'assurer que ces espaces de modules globaux soient compatibles aux modèles locaux, on dispose du théorème DGK. La nature de la singularité en un arc $\g(t)$ est alors contenue dans la partie de dimension finie; si de plus on se place sur un corps fini, il est possible, à partir de ces morceaux de dimension finie de définir une \og fonction locale\fg~ et Bouthier-Ngô-Sakellaridis montrent que la fonction locale est compatible à la \og fonction globale\fg~ \cite[1.21-2.5]{BNS}. De plus, on s'attend à ce que cette \og fonction locale\fg~ provienne d'un certain faisceau pervers sur l'espace d'arcs. 

Cela a été notre motivation originale, mais à la lumière de l'énoncé de globalisation, une telle théorie ne semble pas exister (voir \ref{details} pour une discussion plus spécifique).
La principale raison pour cela est que bien que les voisinages formels soient les mêmes, lorsque l'on regarde les schémas affines correspondants, le voisinage formel épointé n'a pas les mêmes singularités que l'ouvert correspondant sur $\cL X^{\leq d}$. On peut donc espérer un énoncé du type \og le pullback du IC sur $\cL X^{\leq d}$ donne le IC (à décalage près) d'un voisinage formel de dimension finie.\fg.

\bigskip

Passons en revue l'organisation de l'article; dans la première section, on fait des rappels sur la décomposition de Weierstrass et on étudie la classe des morphismes pro-lisses et Weierstrass qui interviennent dans la preuve de \ref{thm1}.
Dans la deuxième section, on introduit les schémas de type (S) et on fait une étude géométrique de leur propriété. La troisième section montre ensuite l'énoncé de globalisation \ref{thm1} et enfin la dernière section établit les énoncés cohomologiques de finitude \ref{thm2} et \ref{thm3}.
\medskip

A la suite d'une série d'exposés au séminaire Drinfeld sur notre travail, B.C. Ngô a écrit des notes \cite{NWei} sur le théorème \ref{thm1}, en même temps qu'il nous a signalé une erreur dans la preuve de l'énoncé. Cette nouvelle version a grandement bénéficié des clarifications qu'il a faites à cette occasion, nous lui en sommes très reconnaissants.
\medskip

Ce travail a débuté en collaboration avec David Kazhdan, avec qui nous avons eu de nombreuses discussions, mais le projet initial a eu des modifications substantielles et pour le résultat final, Kazhdan a décliné d'être auteur. Toutefois, notre dette à son égard reste significative.

Nous tenons  à exprimer nos plus vifs remerciements à Peter Scholze pour sa relecture attentive du manuscrit ainsi que pour ses nombreux commentaires éclairés. Nous remercions également Thomas Bitoun, Ofer Gabber, Gérard Laumon pour les multiples discussions et questions que nous avons eu au sujet de ce travail. Enfin, nous remercions Bhargav Bhatt, Julien Sebag, David Bourqui ainsi que K\k{e}stutis \v{C}esnavi\v{c}ius pour l'amélioration de la forme. Nous tenons également à remercier l'Université Hébraïque de Jerusalem pour les très bonnes conditions de travail.

\section{Division de Weierstrass et morphismes pro-lisses}
\subsection{Le théorème de préparation de Weierstrass}\label{Artin}
Soit $A$ local complet, $\km$ son idéal maximal, de corps résiduel $k$.
On considère $f\in A[[t]]$ tel qu'en réduction $\bar{f}=t^nu(t)$ avec $u(t)\in k[[t]]^{\times}$.
Par le théorème de factorisation de Weierstrass (\cite[§3, n°9, Prop.6]{Bki}), il existe un unique polynôme unitaire $q\in A[t]$ de degré $n$ tel que:
\[
f=qv,
\]
avec $v\in A[[t]]^{\times}$.
De plus, d'après loc.cit., on a un énoncé de division de Weierstrass, pour $f\in A[[t]]$ comme ci-dessus et pour tout $g\in A[[t]]$, il existe un unique couple $(b,r)\in A[[t]]\times A[t]$ avec $\deg(r)\sinf n$, tel que:
\[g=bf+r.
\]
Soit $\Inf_{k}$ la catégorie des $k$-algèbres locales $A$, d'idéal maximal $\km$, de corps résiduel $k$, telles que $\km^m=0$ pour un certain $m\in\NN$. Les morphismes sont ceux de $k$-algèbres. Dans la suite, on appelle les objets de $\Inf_k$ des \textsl{anneaux quasi-artiniens}. En particulier, ces anneaux sont locaux complets, donc satisfont la division de Weierstrass.  Si $A$ est dans $\Inf_k$ et est noethérien, il est artinien.
\subsection{Rappels sur les espaces d'arcs}
Soit un corps $k$ algébriquement clos et un $k$-schéma de type fini $X$. 
Pour tout entier $j\in\NN$, on considère le foncteur $\cL_{j}X$ des arcs tronqués d'ordre $j$, dont les $R$-points, pour une $k$-algèbre $R$, sont les $\cL_{j}X(R)=X(R[t]/(t^{j+1})$. Il est représentable par un $k$-schéma de type fini et pour $j\geq i$ les flèches $\cL_{j}X\rightarrow\cL_{i}X$ sont affines. On considère alors l'espace d'arcs formels:
\[
\cL X:=\varprojlim\limits_{ j\in\NN}\cL_{j}X.
\]
En général, c'est un $k$-schéma non noethérien (à part si $\dim X= 0$, cf. \cite[Thm.2.5.5]{NS} pour un énoncé relatif plus général). D'après un théorème de Bhatt \cite[Thm.1.1]{Bh}, il vérifie la propriété universelle pour toute $k$-algèbre $R$ :
\[
\Hom(\Spec(R),\cL X)=\Hom(\Spec(R[[t]]),X).
\]
Si $X=\ab^{1}$, on a $\cL \ab^{1}=\ab^{\NN}=\Spec(k[x_{i}]_{i\in\NN})$. Pour tout $j\in\NN$, on dispose de flèches de projection :
\[
 f_{j}:\cL X\rightarrow\cL_{j}X.
\]
Si $X$ est lisse, ces flèches sont formellement lisses et surjectives. Dans le cas singulier, les flèches ne sont ni formellement lisses, ni plates, ni même  surjectives et l'étude des singularités de l'espace d'arcs s'avère délicate.
De plus, si on a un morphisme $f:X\rightarrow Y$, on a par fonctorialité un morphisme $\cL f:\cL X\rightarrow\cL Y.$
Si $f$ est étale alors $\cL f$ l'est également et on a un carré cartésien (\cite[sect.3]{ME}):
$$\xymatrix{\cL X\ar[r]\ar[d]&\cL Y\ar[d]\\X\ar[r]&Y}.$$
En particulier si on a un recouvrement de $X$ par des ouverts $U_{i}$, les ouverts $\cL U_{i}$ forment un recouvrement ouvert de $\cL X$.
Dans le cas lisse, la structure de l'espace d'arcs est aisée à étudier:
\begin{lem}\label{lissearc}
Soit $X$ un $k$-schéma lisse connexe de type fini avec $\dim X\geq 1$, alors localement pour la topologie Zariski, $\cll$ est un produit $U\times\ab^{\NN}$ avec $U$ un $k$-schéma lisse de type fini.
\end{lem}
\begin{proof}
On commence par écrire $X=\bigcup U_{i}$ avec des ouverts affines $U_{i}$ étales sur $\ab^n$ avec $n\geq 1$. D'après ci-dessus, on a:
\[
\clui\cong U_i\times_{\ab^{n}}\cLp\ab^{n}\cong U_{i}\times\ab^{\NN}
\]
et les $\clui$ forment un recouvrement ouvert de $\cll$.
\end{proof}
Tout le but est de maintenant comprendre la situation singulière, bien plus complexe. On commence par un théorème sur la structure formelle de l'espace d'arcs. 
Dans la suite, on considère un corps $k$ parfait. Soit un $k$-schéma de type fini $X$ géométriquement réduit et pur. On considère le sous-schéma jacobien $\Jac_{X}$, défini par l'idéal de Fitting $\Fitt_{\dim X}(\Omega_{X})$. On renvoie à \cite{Eis} pour des rappels sur les idéaux de Fitting.
Le support du schéma jacobien est précisément le lieu singulier $X_{sing}$ de $X$ et c'est un fermé strict d'après \cite[Tag. 056V]{Pro}.
Si $X$ est affine, en le plongeant dans un espace affine $\ab^{N}$ défini par des équations $f_{1}=\dots=f_{m}=0$ dans $\ab^{N}$, alors $\Jac_{X}$ est engendré par les mineurs d'ordre $l$ de la matrice jacobienne $(\partial f_{i}/\partial x_{j})_{i,j}$ avec $l=N-m$. on considère l'ouvert des arcs non-dégénérés:
\[
\clbul:=\cL X-\cL X_{sing}.
\]
C'est un ouvert qui n'est pas quasi-compact non-vide comme $X_{sing}$ est un fermé strict.
On rappelle le théorème de Drinfeld \cite{Dr} et Grinberg-Kazhdan \cite{GK} :
\begin{theorem}\label{gk}
Soit un corps $k$ parfait, un $k$-schéma type fini $X$ géométriquement réduit. Soit $\g(t)\in\clbul(k)$, tel que $\g(0)$ ne soit pas dans une composante isolée, alors il existe un $k$-schéma de type fini $Y$ et  $y\in Y(k)$ tel que l'on a un isomorphisme de voisinages formels:
\[
\cll_{\g(t)}^{\wedge}\simeq (Y\times\ab^{\NN})^{\wedge}_{(y,0)}.
\]
\end{theorem}
\brems
\remi

Si $\g(0)$ est dans une composante isolée de $X$, alors $\g(t)$ est un point isolé de $\clbul$ et le voisinage formel s'identifie à $\Spec(k)$.
\remi
A la suite de \cite{BNS}, on appelle la paire $(Y,y)$ \textit{un modèle formel de dimension finie} en $\g(t)$. Il s'agit maintenant de voir comment l'on peut globaliser ce théorème.
\remi
Un modèle formel de type fini n'est en général ni irréductible, ni même réduit.
D'après Bourqui-Sebag \cite{BoS}, le théorème de Drinfeld-Grinberg-Kazhdan ne s'étend pas aux arcs dégénérés $\g\in\cll_{sing}\subset\cll$.
\erems
Ce modèle formel de type fini peut se calculer de manière très explicite, l'ingrédient clé est le théorème de préparation de Weierstrass.
Voyons comment on utilise ce théorème dans un exemple précis tiré de \cite{Dr}. 

\bex
Soit $X$ l'hypersurface $yx_{n+1}+f(x_1,\dots,x_n)=0$ et l'arc $\g_{0}(t)$ et $x^0_{n+1}(t)=t$, $y=x_1^{0}(t)=\dots=x_n^{0}(t)=0$.
Alors, le modèle formel $Y$ est donné par l'hypersurface $f(x_1,\dots,x_n)=0$ et $y=0$. En effet, par le théorème de division de Weierstrass, pour tout anneau $A\in\Inf_k$, une $A$-déformation de $x^0_{n+1}(t)=t$ est de la forme $x_{n+1}(t)=(t-\al)u(t)$ avec $\al\in\km$ et $u\in A[[t]]$ inversible. Si l'on se fixe $\al$, $u$, $x_1(t),\dots,x_n(t)\in\km[[t]]$, il y a au plus un $y(t)\in\km[[t]]$ tel que $y(t)x_{n+1}(t)+f(x_1(t),\dots,x_n(t))=0$ et $y(t)$ existe si et seulement si $f(x_1(\al),\dots,x_n(\al))=0$.
\eex
On veut obtenir un énoncé global, qui ne fait pas intervenir les complétions formelles. L'inconvénient est que l'on ne dispose pas de théorème de division de Weierstrass pour des anneaux locaux henséliens, cela nous amène donc à étudier une nouvelle classe de morphismes.
\subsection{Morphismes pro-lisses}\label{Dpro}
Soit $f:Z\rightarrow S$ un morphisme de schémas. On dit que:
\benumr
	\item 
	$f$ est \textsl{pro-lisse} (resp. \textsl{pro-étale)} s'il s'écrit comme une limite projective filtrante de morphismes lisses (resp. étales) avec des morphismes de transition affines.
	\item
	$f$ est \textsl{placide} s'il s'écrit comme une limite projective filtrante $Z\simeq\varprojlim\limits Z_{i}$ où les $Z_{i}$ sont de présentation finie sur $S$ et où les morphismes de transition sont affines, lisses.
\eenum
\brems\label{tfpro}
\remi 
	Il est crucial de remarquer que pour un morphisme pro-lisse, les morphismes de transition sont arbitraires. En revanche, pour un morphisme pro-étale, les morphismes de transition sont automatiquement étales (\cite[Tag. 02GW]{Pro}).
	\remi
	Un morphisme $\Spec(B)\rightarrow\Spec(A)$ pro-lisse de type fini est lisse. En effet, dans ce cas $B$ admet un nombre fini de générateurs et en écrivant $B\simeq\colim B_{\al}$ avec $B_{\al}$ des $A$-algèbres lisses, on obtient que $B\simeq B_{\al}$ pour $\al$ assez grand.
	\remi
	Comme une limite inductive filtrante de modules plats est plate, on en déduit qu'un morphisme pro-lisse entre schémas $f:Z\rightarrow S$ est plat et comme d'après \cite[II.1.2.3.4]{Ill}, le complexe cotangent commute aux limites inductives filtrantes d'algèbres, on en déduit que le complexe cotangent $L_{Z/S}$ est concentré en degré zéro et plat sur $Z$.
		\erems
Les morphismes pro-lisses et placides sont stables par changement de base arbitraire et par composition. En effet, traitons le cas d'un morphisme pro-lisse, les autres cas sont analogues.
On considère alors $Z\rightarrow Y\rightarrow X$ une composée de morphismes pro-lisses.
On a donc $Z\simeq\varprojlim Z_{i}$ avec $Z_{i}$ lisse de présentation finie sur $Y$. Maintenant, comme $Y$ est pro-lisse sur $X$, on a $Y\simeq\varprojlim Y_{j}$ avec les $Y_{j}$ lisses de présentation finie sur $X$ et il résulte alors \cite[Tag. 0C0C, 01ZL]{Pro} que chaque $Z_{i}$ se descend en un schéma lisse de présentation finie sur un des $Y_{j}$, le résultat suit.
\medskip

\brem La réciproque est malheureusement fausse d'après un contre-exemple de Gabber, rédigé par Bhatt \cite{Gab1}. De plus, l'exemple rédigé par Bhatt pourrait amener à croire qu'il s'agit d'une pathologie typique de la caractéristique $p$. Néanmoins, Gabber nous a communiqué que le résultat reste faux en caractéristique zéro. En remplaçant la perfection de l'algèbre $B_{i}$ de \cite{Gab1} par l'algèbre où l'on ajoute les racines carrés, on obtient à nouveau une algèbre non-réduite dont le complexe cotangent est plat et placé en degré zéro.
\erem
On s'intéresse dans la section suivante aux morphismes pro-lisses entre schémas noethériens.
\subsection{Morphismes réguliers}
On rappelle qu'un morphisme \textsl{régulier} est un morphisme plat à fibres géométriquement régulières et que si $A$ est un anneau excellent, une des propriétés est que pour tout $\kp\in\Spec(A)$, le morphisme de complétion $A_{\kp}\rightarrow \hat{A}_{\kp}$ est régulier.
De plus,  si $A$ est excellent, il en est de même de toute $A$-algèbre de type fini. Les exemples typiques d'anneaux excellents qui apparaîtront ici sont donnés par les $k$-algèbres de type fini et les anneaux locaux complets noethériens. On renvoie à \cite[Tag. 07QS]{Pro} pour plus de détails. 
\medskip

D'après le théorème de Popescu (\cite{P1},\cite{P2},\cite{P3}), tout
morphisme régulier $f:\Spec(B)\rightarrow\Spec(A)$ entre  schémas noethériens est pro-lisse. En particulier, si $Y$ est un schéma excellent et $y\in Y$, le morphisme de complétion:
	\[
	\Spec(\cO_{Y,y}^{\wedge})\rightarrow\Spec(\cO_{Y,y})
	\]
	est pro-lisse. On remarque ainsi qu'un morphisme pro-lisse n'est pas nécessairement formellement lisse.
A l'aide du théorème de Popescu, on peut clarifier la relation entre régulier et pro-lisse dans le cas noethérien:
\begin{theorem}\label{pop}
Soit un morphisme  $f:\tilde{Y}\rightarrow Y$ entre schémas affines noethériens, on a les équivalences suivantes:
\benumr
	\item 
	$f$ est régulier.
	\item
	$f$ est pro-lisse.
	\item
	$f$ est plat et le complexe cotangent $L_{\tilde{Y}/Y}$ est concentré en degré zéro et plat sur $\tilde{Y}$.
\eenum
\end{theorem}

\begin{proof}
$(i)\Rightarrow (ii)$ est l'objet du théorème de Popescu, d'après \ref{tfpro}, on a l'implication $(ii)\Rightarrow (iii)$ et $(iii)\Rightarrow (i)$ se déduit d'un théorème d'André \cite[Th.30, p. 331]{Andre}.
\end{proof}
\subsection{Morphismes de Weierstrass}
Soit $k$ algébriquement clos, $f:Z\rightarrow S$ un morphisme entre $k$-schémas, soit $x\in Z(k)$ et $y=f(x)$.
Soit $\hat{f}_{x}:\Def_{Z,x}\rightarrow\Def_{Y,y}$ le morphisme au niveau des foncteurs de déformation, on dit que $\hat{f}_{x}$ est \textsl{topologiquement de type fini} s'il existe un entier $r\in\NN$ et un diagramme commutatif:
$$\xymatrix{\Def_{Z,x}\ar[dr]\ar[r]&\Def_{\ab^{r}_{Y},y}\ar[d]\\&\Def_{Y,y}}$$
où la flèche horizontale est une immersion fermée (i.e représentable par une immersion fermée après tiré-en-arrière par tout $\Spec(R)$ avec $R$ quasi-artinien) et la flèche verticale de droite est induite par la projection.
En utilisant le fait que les foncteurs de déformation sont pro-représentables par les anneaux locaux complétés, cela se réinterprète en disant qu'on a une surjection:
\[\cO_{Y,y}^{\wedge}[[t_1,\dots,t_{r}]]\twoheadrightarrow \cO_{Z,z}^{\wedge}.\]

\begin{defi}\label{DWei}
Soit $k$ algébriquement clos, $f:Z\rightarrow S$ un morphisme entre $k$-schémas. On dit que :
\benumr
	\item 
	$f$ est  \textsl{Weierstrass}, si $f$ est pro-lisse et pour tout $k$-point $x\in Z(k)$, $\hat{f}_{x}$ est formellement lisse et topologiquement de type fini.
	\item
	$f$ est un \textsl{isomorphisme formel} si pour tout $k$-point $x\in Z(k)$,	$\hat{f}_{x}$ est un isomorphisme.
	\eenum
\end{defi}
\brem\label{rem-def}
Comme  les foncteurs de déformation commutent aux produits et que la formelle lissité est stable par changement de base et composition, les morphismes de Weierstrass et les isomorphismes formels sont stables par changement de base arbitraire et composition ainsi que par produit direct.
\erem

\begin{lem}\label{D-Top}
On conserve les notations de la définition précédente, alors si $\hat{f}_{x}$ est topologiquement de type fini et formellement lisse, on a un isomorphisme (non-canonique) de foncteurs:
\[\Def_{\ab^{r}_{Y,y}}\simeq\Def_{Z,x}\]
pour un certain entier $r\in\NN$.
\end{lem}
\brem\label{rem-top}
Comme les anneaux locaux complétés pro-représentent les foncteurs de déformations, il résulte du lemme que la condition de \ref{DWei}.1. est équivalente à ce que le morphisme:
\[\hat{\eta}_{x}:\cO^{\wedge}_{Y,y}\rightarrow\cO^{\wedge}_{Z,x}\]
soit formellement lisse et induise un isomorphisme $\cO^{\wedge}_{Y,y}[[t_1,...,t_{r}]]\simeq\cO_{Z,x}^{\wedge}$.
\erem
\begin{proof}
Notons $A=\cO_{Y,y}$, $B=\cO_{Z,x}$, $\km_{A}$, $\km_{B}$ leurs idéaux maximaux respectifs et $\wA$ (resp. $\wB$) la complétion $\km_{A}$-adique (resp. $\km_{B}$-adique).
Pour tout $n\in\NN$, on définit alors $\widehat{\km_{A}^{n}}=\Ker(\wA\rightarrow A/\km_{A}^n)$ et de même $\widehat{\km_{B}^{n}}$.
En général, on a $\widehat{\km_{A}^{n}}\neq\km_{A}^{n}\wA$ et $\wA$ n'est pas $\km_{A}$-adiquement complet, il est donc plus naturel de considérer sur $\wA$ la topologie induite par la filtration $F^{\bullet}_{A}=(\widehat{\km_{A}^{n}})$, pour laquelle $\wA$ est tautologiquement complet. En particulier, on a:
\begin{equation}
\forall~n\in\NN, \wA/\widehat{\km_{A}^{n}}\cong A/\km_{A}^{n}, \widehat{\km_{A}^{n}}/\widehat{\km_{A}^{n+1}}\cong \km_{A}^{n}/\km_{A}^{n+1}.
\label{fil}
\end{equation}
et pareillement pour $\wB$. Comme la flèche $\hat{f}_{x}$ est topologiquement de type fini, on a une surjection d'anneaux locaux:
\[\widehat{A}[[t_1,\dots,t_r]]\rightarrow \wB.\]
On peut donc choisir une base \textsl{finie} $x_1,\dots, x_s$ de $\bt^{*}_{\wB/\wA}=\widehat{\km_{B}}/(\widehat{\km_{B}^{2}}+\widehat{\km_{A}}B)\cong\km_{B}/(\km_{B}^{2}+\km_{A}B)$ d'après \eqref{fil} ($\widehat{\km_{B}}/\widehat{\km^{2}_{B}}$ s'envoie surjectivement sur $\bt^{*}_{\wB/\wA}$). On pose  $T:=A[X_1,\dots X_{s}]$, soit $\km_{T}$ sont idéal maximal et $\wT=\wA[[X_1,\dots X_{s}]]$. On a un morphisme de $A$-algèbres locales:
\[\wB\rightarrow T/(\km_{T}^{2}+\km_{A}T)\]
On utilise alors la formelle lissité pour relever d'abord à $T/\km_{T}^2$ puis en un morphisme
\[u:\wB\rightarrow \wT\]
qui induit un isomorphisme entre $\bt_{\wB/\wA}^{*}$ et $\bt_{\wT/\wA}^{*}$. En particulier, une chasse au daigramme analogue à \cite[Lem. 1.1]{Sch} donne que la flèche sur les gradués $\gr(u)$ est surjective (pour les filtrations $F_{B}^{\bullet}$ et $F_{T}^{\bullet}$), donc $u$ est surjectif d'après \cite[Ch. III, n°2, Cor.2]{Bkifil}.
De plus, si on choisit $y_i\in \wB$ tel que $u(y_i)=X_i$, on peut poser $v(X_i)=y_i$ et construire une flèche de $\wA$-algèbres locales $v:\wT\rightarrow \wB$ telle que $u\circ v=\Id_{\wT}$. Ainsi, $v$ est injective et induit un isomorphisme sur les cotangents, donc $v$ est aussi surjective, donc c'est un isomorphisme.
\end{proof}

\begin{prop}\label{wei1}
Soit un morphisme $f:Z\rightarrow S$  entre $k$-schémas $(\aleph_0)$-placides affines avec des morphismes de transition surjectifs, on suppose qu'en tout $k$-point $x\in Z(k)$, $\hat{f}_{x}$ est formellement lisse, alors $f$ est pro-lisse. 
Si de plus, pour tout $x\in Z(k)$, $\hat{f}_{x}$ est topologiquement de type fini, il est donc de Weierstrass.
\end{prop}
\brem Par $(\aleph_0)$-placide, on entend que le système projectif qui intervient est dénombrable.
\erem
\begin{proof}
On écrit $Z=\Spec(A)$ avec $A=\varinjlim\limits A_{i}$ et $S=\Spec(B)$ avec $B=\varinjlim B_{j}$ où les $A_{i}, B_{j}$ sont des $k$-algèbres de type fini et les morphismes de transition sont lisses surjectifs. Fixons $j$, alors le morphisme $Z\rightarrow S_{j}:=\Spec(B_{j})$ se factorise par un $Z_{i}:=\Spec(A_{i})$.
On a alors un diagramme commutatif:
$$\xymatrix{Z\ar[d]_{p_{i}}\ar[r]^{f}&S\ar[d]^{q_{j}}\\Z_{i}\ar[r]& S_{j}}$$
où les flèches verticales sont formellement lisses et surjectives, comme les systèmes inductifs sont dénombrables.
Soit $x\in Z(k)$, on pose $y=f(x)$, $x_{i}:=p_{i}(x)$ et $y_{j}=q_{j}(x)$, alors on a un diagramme au niveau des complétés formels:
$$\xymatrix{\Def_{Z,x}\ar[d]\ar[r]^{\hat{f}_{x}}&\Def_{S,y}\ar[d]\\\Def_{Z_{i},x_{i}}\ar[r]& \Def_{S_{j},y_{j}}}$$
où cette fois $\hat{f}_{x}$ est également formellement lisse et les verticales sont toujours formellement lisses et surjectives.
On déduit donc que la flèche :
\[
\Def_{Z_{i},x_{i}}\rightarrow \Def_{S_{j},y_{j}}
\]
est formellement lisse et comme $Z_{i}$ et $S_{j}$ sont des $k$-schémas de type fini, d'après \cite[Tag. 02HY]{Pro}, la flèche
\[
f_{ij}:Z_{i}\rightarrow S_{j}
\]
est lisse en $x_{i}$. Comme par hypothèse $Z\rightarrow Z_{i}$ est surjective, $f_{ij}$ est lisse en tout $k$-point de $Z_{i}$  et comme les points fermés sont denses dans $Z_i$ et le lieu lisse ouvert, $f_{ij}$ est lisse.
On obtient alors que $f$ est pro-lisse comme il s'écrit comme limite projective filtrante des morphismes $Z_{i}\times_{S_{j}}S\rightarrow S.$
\end{proof}
\brem A partir de maintenant, tous les systèmes projectifs qui apparaîtront seront dénombrables, sauf mention explicite.
\erem

La terminologie de morphisme de Weierstrass se justifie par ce qui suit.
Pour tout entier $d\geq 1$, on considère le $k$-schéma $\cQ_{d}$ qui classifie les polynômes unitaires de degré $d$, l'ouvert $\cQ_{d}^{\flat}\subset\cQ_d$ des polynômes de $\cQ_d$ dont le terme constant est inversible et  le $k$-schéma $A_{d}$ qui classifie les polynômes de degré au plus $d-1$. On définit l'ouvert:
\[
 \cL\ab^{\leq d}:=\ev_{d}^{-1}(\cL_{d}\ab^{1}-\{0\}),
\]
avec $\ev_{d}:\cL\ab^{1}\rightarrow\cL_{d}\ab^{1}$ .
On considère alors le morphisme:
\begin{equation}
\al_{d}:\cQ_{d}\times\cL\bG_{m}\rightarrow\cL\ab^{\leq d}
\label{aldef}
\end{equation}
donné par $(q,u)\mapsto qu$.

\begin{prop}\label{alw}
Le morphisme $\al_{d}$ est  $\aleph_0$-Weierstrass, affine, surjectif et induit un isomorphisme formel en les $k$-points de la forme $(t^d,u)$.
\end{prop}
\begin{proof}
Comme la source est affine et le but est un ouvert d'un schéma affine, $\al_d$ est affine d'après \cite[Tag. 01SG]{Pro}.
Les schémas sont clairement $\aleph_0$-placides et affines et la flèche est surjective sur les $K$-points, pour tout corps $K$, donc surjective.
D'après la  proposition \ref{wei1}, il faut donc  démontrer que $\al_{d}$ est formellement lisse et topologiquement de type fini sur les complétés formels des  $k$-points.
Il suffit d'étudier le foncteur des déformations, soit $A\in\Inf_{k}$, d'idéal maximal $\km_{A}$ et de corps résiduel $k$.
Soit $\g_{0}=t^{d}u\in k[[t]]$, $u\in k[[t]]^{\times}$ et $\g\in A[[t]]$ qui se réduit  modulo $\km_{A}$ sur $\g_{0}$.
Par le théorème de division de Weierstrass, on a alors:
\[
\g(t)=qv,
\]
avec $v\in A[[t]]^{\times}$ et $q\in \cQ_{d}(A)$. On obtient ainsi l'isomorphisme au niveau des complétés formels en $\g_{0}$.
Il s'agit de voir l'assertion en les autres points.

Soit $y\in\cL\ab^{\leq d}(A)$ de réduction $y_{0}=t^{e}v_{0}$ avec $v_{0}\in k[[t]]^{\times}$ et $e\leq d$. Montrons que la fibre du morphisme $\al_d(A)$ est donnée par $\cQ^{\flat}_{d-e}(A)$. 
On commence par appliquer le théorème de Weierstrass à $y$ et on écrit $y=pv$ avec $p\in\cQ_{e}(A)$ et $v\in A[[t]]^{\times}$. La fibre au-dessus de $y$ est alors donnée par les paires $(q,u)\in(\cQ_{d}\times\cL\bG_{m})(A)$ telles que:
\[
q=pvu^{-1}.
\]
Comme on a de plus, $A[t]/(p)\cong A[[t]]/(p)$, puisque les deux $A$-modules sont libres de même rang, on en déduit que
$p\vert q$ et $q=pp'$, avec $p'\in\cQ_{d-e}(A)$. De plus, en réduisant modulo $t$ l'égalité $qu=pv$, on obtient que $p'\in A[[t]]^{\times}$,
on obtient alors que $u$ est déterminé uniquement par:
\[
u=v(p')^{-1},
\]
ce qu'on voulait.
\end{proof}
On introduit la stratification naturelle de $\cQ_{d}$. Pour $0\leq e\leq d$, on considère le fermé $\cQ_{d,e}\subset\cQ_{d}$ défini par:
\[
\cQ_{d,e}:=\{q\in\cQ_{d}~\vert~ t^{e}\vert q\}.
\]
En particulier, on a $\cQ_{d,0}=\cQ_{d}$ et une suite décroissante de fermés:
\[
\cQ_{d,0}=\cQ_{d}\supset\cQ_{d,1}\supset\dots\supset\cQ_{d,d}=\{t^d\}.
\]
La multiplication par $t^{e}$ induit un isomorphisme:
\begin{equation}
t^{e}:\cQ_{d-e}\cong\cQ_{d,e}.
\label{eq0}
\end{equation}
 Pour $0\leq e\leq d$, on considère les ouverts:
\[
\cQ^{\flat}_{d,e}:=\{q\in\cQ_{d,e}~\vert~ a_{e}\in\bG_{m}\}\cong\bG_{m}\times\ab^{d-1-e}.
\]
avec la convention $\cQ^{\flat}_{d,d}=\cQ_{d,d}$ et $\cQ_{d}^{\flat}:=\cQ_{d,0}^{\flat}$.
Le morphisme $\al_d$ admet une description simple sur chacune des strates où l'on fixe la valuation:
\begin{prop}\label{alreg}
Pour tout $0\leq e\leq d$, on a un isomorphisme canonique:
\begin{equation}
\cQ^{\flat}_{d-e}\times\cL\ab^{=e}\simeq\al_{d}^{-1}(\cL\ab^{=e}).
\label{str}
\end{equation}
de telle sorte que la composée $\cQ^{\flat}_{d-e}\times\cL\ab^{=e}\ra \cL\ab^{=e}$ s'identifie à la projection suivant le second facteur.
En particulier, le morphisme est régulier.
\end{prop}
\begin{proof}
Le localement fermé  $\cL\ab^{=e}$ consiste en les séries de la forme $t^{e}v$ avec $v\in\cL\bG_m$.
On considère la flèche canonique:
\[
\cQ^{\flat}_{d-e}\times\cL\ab^{=e}\ra\al_{d}^{-1}(\cL\ab^{=e}).
\]
 définie par $(p, v)\mapsto (t^{e}p,p^{-1}v)$ où $p\in\clp \bG_m$ comme $p(0)\in\bG_m$. Montrons que c'est un isomorphisme.
Soit $(q,u)\in \al_{d}^{-1}(\cL\ab^{=e})$, on a alors :
\[
qu=t^e v,
\]
 pour un certain $v\in\cL\bG_m$, d'où $t^{e}vu^{-1}=q$, d'où l'on déduit en identifiant les coefficients que $t^e\vert q$ et donc on a:
\[
q=t^e p
\]
avec $p\in\cQ^{\flat}_{d-e}$, comme souhaité.
D'après la \Cref{alw}, on a déjà que le morphisme est pro-lisse, donc plat et on vient de voir que les fibres de $\al_d$ sont données par des schémas affines  lisses, donc le morphisme est régulier.
\end{proof}
Le morphisme $\al_{d}:Z\rightarrow S$ avec $Z:=\cQ_{d}\times\cL\bG_m$ et $S:=\cL\ab^{\leq d}$ induit un morphisme au niveau des espaces tangents:
\[
d\al_d:TZ\rightarrow TS\times_{S}Z.
\]
avec $TZ:=\cQ_{d}\times\cL\bG_m\times A_{d}\times\cL\ab^1$ et $TS\times_{S}Z=\cQ_{d}\times\cL\bG_m\times\cL\ab^1$. A un quadruplet $(q,u,a,v)\in TZ$, on associe:
\[
d\al_d(q,u,a,v)=(q,u,ua+qv).
\]
En prenant le paramètre $u=1$, on obtient ainsi le morphisme:
\begin{equation}
\beta_{d}:\cQ_{d}\times A_{d}\times\cL\ab^{1}\rightarrow \cQ_{d}\times\cL\ab^{1},
\label{beta}
\end{equation}
donné par $(q,v,\xi)\mapsto (q,v+q\xi)$.

\begin{prop}\label{betw}
Le morphisme $\beta_{d}$ est  Weierstrass, affine, surjectif et induit un isomorphisme formel aux $k$-points de la forme $(t^d,v,\xi)$.
\end{prop}
\begin{proof}
On a déjà que $\beta_{d}$ est affine, comme morphisme entre schémas affines.
D'après la \Cref{alw}, $\al_{d}$ est Weierstrass, surjectif donc d'après \ref{wei1} et critère infinitésimal, $d\al_{d}$ l'est également et $\beta_{d}$ par changement de base.
\end{proof}

\begin{prop}\label{betreg}
Le long de la stratification $\{\cQ^{\flat}_{d,e}\times\cL\ab^{1}\}_{0\leq e\leq d}$, on a un isomorphisme canonique:
\[
\cQ^{\flat}_{d,e}\times A_{d-e}\times\cL\ab^{1}\simeq\beta_{d}^{-1}(\cQ^{\flat}_{d,e}\times\cL\ab^{1}).
\]
\end{prop}
\begin{proof}
A nouveau, la proposition se déduit de \ref{alreg} par différentiation et changement de base ainsi que du fait que la stratification de $\cL\ab^{\leq d}$ par les $\cL\ab^{=e}$ induit la stratification correspondante sur $\cQ_{d}\times\cL\ab^{1}$.
\end{proof}

\section{Briques non-noethériennes}
Dans cette section, on considère une certaine classe d'anneaux qui vont apparaître de manière cruciale dans le \Cref{fonda}.
\subsection{Schémas décents}\label{schemas-dec}
Soit un schéma  $S$, on considère le faisceau d'idéaux $\nil_{\infty}(\cO_S)$ dont les sections sur un ouvert sont données par :
\[
\nil_{\infty}(\cO_{S}(U)):=\bigcap\limits_{n\geq 0}\bigcap\limits_{t\in U}\Ker(\cO_{S}(U)\rightarrow \cO_{U,t}/\km_{t}^{n}).
\]
On note $i_{dec}:S_{dec}\hookrightarrow S$, le sous-schéma fermé défini par ce faisceau d'idéaux et on dit que $S$ est \textsl{décent }si $i_{dec}$ est un isomorphisme. Un anneau $A$ est \textsl{décent} si $\Spec(A)$ l'est.  
Le schéma $S_{dec}$ vérifie la même propriété d'unicité que le réduit (analogue à \cite[Tag. 01J3]{Pro}), à savoir qu'il est l'unique schéma décent avec le même espace topologique que $S$.
De plus, par réduction au cas affine, on obtient que pour tout morphisme de schémas $W\rightarrow S$ avec $W$ décent, la flèche se factorise par $S_{dec}$.

Si $S=\Spec(A)$, alors $x\in\nil_{\infty}(A)$ si pour tout $\kp\in\Spec(A)$, $x$ est dans le noyau de $A\rightarrow A_{\kp}/\kp^{n}A_{\kp}$ pour tout $n$. On en déduit ainsi:
\[
x\notin\nil_{\infty}(A)\Longleftrightarrow \exists~ \nu:A\rightarrow R~~ \text{tel que} ~~\nu(x)\neq 0~~ \text{et}~~ R\in\Inf_{k}.
\] 
Des exemples d'anneaux décents sont fournis par les familles suivantes:
\bitem
\m
	Si $A$ est réduit, car $\nil(A)=\bigcap\limits_{\kp\in\Spec(A)}\kp=\{0\}$.
\m
	Si $A$ est noethérien, en effet d'après \cite[Tag. 00L9]{Pro} il s'injecte  dans le produit $\prod\limits_{\kp\in\Ass(A)}A_{\kp}$ où $\Ass(A)$ désigne l'ensemble, fini car $A$ est noethérien, des idéaux premiers minimaux et les anneaux semi-locaux noethériens sont décents par le théorème d'intersection de Krull (\cite[Tag. 00IP]{Pro}).
\m
	Tout anneau local complet est décent.
	\eitem
	Un exemple typique d'anneau indécent est de considérer $B:=A/tA$ où $A:=\bigcup\limits_{n\geq 0} A[[t^{\frac{1}{n}}]]$. L'idéal maximal  de $B$ vérifie $\km=\km^2$ et $\km\subset\nil_{\infty}(B)$.
	\medskip

Une immersion fermée $i:S_{0}\hookrightarrow S$ est dite \textsl{décente} si le faisceau d'idéaux $\cI$ qui la définit est contenu dans $\nil_{\infty}(\cO_{S})$. Les immersions décentes sont stables par changement de base. De plus, comme les foncteurs $S_{dec}(R)$ et $S(R)$ sont les mêmes pour tout anneau $R$ décent et donc en particulier pour les anneaux de $\Inf_{k}$, on en déduit que toute immersion décente est un isomorphisme formel.
\begin{lem}\label{limnil}
Soit $S$ un schéma quasi-compact quasi-séparé, on considère une immersion décente $S'\hookrightarrow S$,
alors elle s'écrit comme une limite filtrante (non-nécessairement dénombrable) d'immersions décentes de présentation finie.
\end{lem}
\begin{proof}
D'après \cite[Tag. 09ZP]{Pro}, toute immersion fermée dans un schéma qcqs est une limite d'immersions fermées de présentation finie et les flèches de transition sont elle-mêmes des immersions fermées (voir preuve de loc.cit.).
Ainsi, on a:
\[
 S'\simeq\varprojlim S_{\al}
\]
Comme de plus, $S'\hookrightarrow S$ est décente, il en est de même des $S_{\al}$.
\end{proof}

\subsection{Les schémas $S_d$}\label{sd}
Soit $d\in\NN^*$, on considère le carré cartésien:
\begin{equation}
\xymatrix{S_{d}\ar[d]_{\psi_{d}}\ar[r]&\cQ_{d}\times A_{d}\times\cL\ab^{1}\ar[d]^{\beta_{d}}\\\cQ_{d}\ar[r]&\cQ_{d}\times\cL\ab^{1}}
\label{defSd}
\end{equation}
où la flèche horizontale du bas est donnée par $q\mapsto(q,0)$.
D'après \ref{betw}, le schéma $S_{d}$ est pro-lisse sur $\cQ_{d}$ et est non-noethérien car $\beta_{d}$ n'est pas de type fini.
Il admet une section:
\begin{equation}
\sigma:\cQ_{d}\rightarrow S_{d}
\label{sigma}
\end{equation}
donné par $q\mapsto (q,0,0)$.
Décrivons-le de manière explicite pour $d=1$.
Le schéma $\cQ_{1}\times A_{1}\times\cL\ab^{1}$ classifie les triplets $(q,v,\xi)$ avec $q=t+a$ le polynôme universel de degré un et $\cQ_{1}:=\Spec(k[a])$. 
On écrit $\cL\ab^{1}:=\Spec(k[\xi_{0},\xi_{1}\dots])$. Les équations pour $S_{1}$ sont alors données par:
\begin{align}
    v+a\xi_{0}& = 0& \\
   \xi_{0}+a\xi_{1}& =0 & \\
	\xi_{1}+a\xi_{2}& = 0& \\
	\dots
\end{align}
avec $\xi=(\xi)_{i\in\NN}$.
En particulier, la fibre au-dessus de $S_{1}$ pour $a\neq0$ est isomorphe à $\ab^{1}$ et la fibre au-dessus de 0 est réduite à un point.
Il résulte de ces équations que pour tout $n\in\NN$, $a^{n}$ divise tous les $\xi_{i}$ et $v$.
En particulier, on trouve que la complétion formelle en 0 du schéma $S_{1}$ est isomorphe à $k[[a]]$.

Plus généralement, on vérifie immédiatement que la fibre en zéro $\psi_{d}^{-1}(0)$  se réduit à 0 et il résulte de \ref{betw}, par changement de base, que la flèche $\psi_{d}:S_{d}\rightarrow\cQ_{d}$ induit un isomorphisme sur les complétions en zéro.
 \medskip

Etudions les schémas $S_{d}$ le long d'une stratification.
\begin{lem}\label{stratD}
Soit $d\in\NN^*$, alors $S_{d}$ admet une stratification finie constructible par $d+1$ schémas de type fini.
Plus précisément, pour tout $e\in\llbracket0,d\rrbracket$, $\psi_{d}^{-1}(\cQ_{d,e}^{\flat})$ est un $k$-schéma de type fini.
\end{lem}
On rappelle que $\cQ_{d,e}^{\flat}:=\{q\in\cQ_{d}\vert, t^{e}\vert q, a_{e}\in\bG_{m}\}$ avec la convention que $\cQ_{d,d}^{\flat}=\{t^d\}$.
\begin{proof}
Cela se déduit par changement de base de \ref{betreg}.
\end{proof}

\subsection{Schémas de type $(S)$}\label{(S)} 
Soit $d\in\NN$ et un $k$-schéma qcqs $T$, on dit qu'il est \textsl{de type} $(S_d)$ s'il admet une partition :
\[T=\coprod\limits_{i=0}^{d} T_{i}\]
par des sous-schémas localement fermés non-vides, de telle sorte que:
\benumr
	\item 
	Pour tout $i\leq d$, $\bigcup\limits_{k\leq i} T_{k}$ est ouvert quasi-compact dans $T$.
\item
Pour tout $i$, les $T_i$ sont des $k$-schémas de type fini.
\end{enumerate}
On dit qu'il est \textsl{de type} $(S)$ s'il est de classe $(S_{d})$ pour un certain $d\in\NN$. De plus, on appelle $(S)$-\textsl{stratification}, une stratification qui vérifie les deux conditions ci-dessus.
Dans ce cas, d'après \cite[Tag. 05GG]{Pro}, pour tout $i$, l'immersion $T_{i}\hookrightarrow T$ est de type fini.

Si $d=0$, les schémas de type $(S_0)$ sont simplement les $k$-schémas de type fini.
Si $T$ est de type $(S_d)$ pour un certain $d\in\NN^*$, il admet un fermé $F\hookrightarrow T$ qui est un $k$-schéma de type fini de complémentaire $V$ quasi-compact de type $(S_{d-1})$:
\begin{equation}
T=V\sqcup F.
\label{ndecomp}
\end{equation}

Soit $A$ un anneau muni de la topologie $I$-adique pour un idéal de type fini $I\subset A$ et un $A$-module. On définit $M^{\wedge}:=\varprojlim M/I^nM$ la complétion $I$-adique de $M$. 
Il résulte de \cite[00M9]{Pro} que l'on a:
\begin{equation}
I^nA^{\wedge}\cong(I^{n})^{\wedge}~~ \text{et}~~ A^{\wedge}/I^nA^{\wedge}\cong A/I^{n}.
\label{ad1}
\end{equation}
En particulier, $A^{\wedge}$ est $I$-adiquement complet.
\begin{lem}\label{dim-comp}
On considère une décomposition telle que dans \eqref{ndecomp}, on a les propriétés suivantes:
\benumr
	\item 
Les schémas de type $(S)$ ont un espace topologique noethérien et sont de dimension de Krull finie.
\item
Tout sous-schéma localement fermé d'un schéma de type $(S)$ est de type $(S)$.
\item
On conserve les notations de \eqref{ndecomp}. L'immersion fermée $F\hookrightarrow T$ est de type fini et si $T$ est affine, alors le schéma affine sous-jacent à la complétion formelle de $T$ le long de $F$ est noethérien excellent.
\eenum
\end{lem}
\bpf
(i) Comme un espace topologique qui admet une stratification finie par des espaces topologiques noethériens de dimension finie est lui-même de noethérien de dimension finie (\cite[Tag. 0053]{Pro}), on en déduit l'assertion pour les schémas de type $(S)$.

(ii) Soit  un localement fermé $H$ d'un schéma $T$ de type $(S)$ et une (S)-stratification $(T_i)$ de $T$. D'après (i), $T$ est noethérien, donc $H$ est aussi noethérien, donc quasi-compact. On considère alors la stratification induite $H\cap T_i$ avec sa structure réduite. La condition (i) est bien vérifiée comme $H$ est quasi-compact. Il s'agit donc de voir que $H\cap T_i$ est un $k$-schéma de type fini. Or, $H\cap T_i\hookrightarrow T_i$ est localement fermé dans un $k$-schéma de type fini, donc lui-même de type fini.

(iii) Comme $F$ est un $k$-schéma de type fini, l'immersion fermée $F\hookrightarrow T=\Spec(A)$ \cite[Tag. 01T8]{Pro} est aussi de type fini, donc l'idéal $I$ qui la définit est de type fini et $A/I$ est noethérien, donc d'après  \cite[Tag. 05GH]{Pro}, la complétion $I$-adique de $A$, $A^{\wedge}$, est un anneau noethérien $I$-adiquement complet. 
\epf
\begin{prop}\label{comp2}
Soit $T$ un schéma de type $(S)$, $y\in T(k)$, alors $\cO_{T,y}^{\wedge}$ est noethérien excellent.
\end{prop}
\bpf
On procède par récurrence sur l'entier $d\in\NN$ tel que $T$ est de type $(S_d)$. 
Si $d=0$, $T$ est un $k$-schéma de type fini et cela se déduit de \cite[Tag. 07QW]{Pro}.
Supposons le résultat vrai pour les schémas de type $(S_d)$ et montrons le résultat pour $(S_{d+1})$.
Soit $T$ de type  $(S_{d+1})$, on écrit alors une décomposition telle que dans \eqref{ndecomp}:
\[T=V\sqcup F\]
où $F$ est un $k$-schéma de type fini, fermé dans $T$ de complémentaire  $V$ de type $(S_d)$.
Si $y\in V(k)$, $\cO_{T,y}^{\wedge}=\cO_{V,y}^{\wedge}$ et c'est clair par hypothèse de récurrence.
Si $y\in F(k)$, alors comme $F$ est un $k$-schéma de type fini, la composée $\Spec(k(y))\hookrightarrow F\hookrightarrow T$ est aussi de type fini d'après \ref{dim-comp}, donc d'après \cite[Tag. 05GH]{Pro}, $\cO_{T,y}^{\wedge}$ est local noethérien complet, donc excellent (\cite[Tag. 07QW]{Pro}).
\epf
\brem
En revanche, en général pour un point quelconque d'un schéma de type $(S)$, la complétion de l'anneau local n'est pas noethérienne en général; il suffit de considérer l'immersion fermée $S_1\hookrightarrow S_2$.
\erem
Pour tout $d\in\NN$, la catégorie des schémas de type $(S_d)$ est stable par produit direct et par morphismes de type fini, ainsi qu'on le voit en prenant respectivement le produit direct et l'image inverse d'une $(S)$-stratification.$\\$
Plus généralement, si $T$ est un schéma de type $(S)$ et $f:T'\rightarrow T$ un morphisme \textsl{$(S)$-stratifié}, i.e qcqs et il existe une $(S)$-stratification $(T_{i})$ de $T$ telle que pour tout $i$, $f_{\vert f^{-1}(T_i)}:f^{-1}(T_i)\rightarrow T_i$ est de type fini, alors $T'$ est aussi de type $(S)$.$\\$
Enfin, pour tout morphisme $(S)$-stratifié $f:X\rightarrow T$ et $g: Y\rightarrow T$ entre schémas de type $(S)$, quitte à raffiner la $(S)$-stratitication pour qu'elle soit adaptée à $f$ et à $g$, on obtient que le produit fibré $X\times_{T}Y$ est aussi de type $(S)$.

Pour tout entier $m\in\NN$, on considère le produit fibré de $m$ copies de $S_d$:
\[
S_{d}^{(m)}:=S_{d}\times_{\cQ_{d}}\dots\times_{\cQ_{d}}S_{d}.
\]
\begin{prop}\label{ouv}
Tout schéma de type fini sur un produit $S_{d_1}^{(m_1)}\times_{k}\dots\times_{k} S_{d_r}^{(m_r)}$ est de type $(S)$.
\end{prop}
\begin{proof}
La catégorie des schémas de type $(S)$ est stable par morphismes de type fini et produits directs. On est donc ramené à montrer l'assertion pour $S_{d}^{(m)}$ et cela se déduit alors de \ref{stratD} et de la stabilité par produit fibré de morphismes $(S)$-stratifiés.
\end{proof}
En fait, dans le théorème \ref{fonda}, ce sont précisément ces schémas là qui apparaissent (et même ceux de présentation finie sur un produit $S_{d_1}^{(m_1)}\times_{k}\dots\times_{k} S_{d_r}^{(m_r)}$).
\subsection{Théorie de la dimension}

\blem\label{jacob}
Tout schéma de type $(S)$ est de Jacobson.
\elem
\bpf
D'après \cite[10.1.2]{EGA}, il s'agit de voir que tout localement fermé de type $(S)$ admet un point fermé. Or d'après \ref{dim-comp}, un localement fermé $H$ d'un schéma de type $(S)$ est aussi de type $(S)$, donc en particulier, contient un ouvert $U$ qui est un $k$-schéma de type fini, donc de Jacobson d'après \cite[10.4.7]{EGA}. En particulier, $U$ admet un point fermé, qui est aussi un point fermé de $H$, ce qu'on voulait.
\epf

\bprop\label{dimouv}
Soit $T$ un schéma de type $(S)$ irréductible de point générique $\eta$, alors $\dim T=\degtr_{k}(k(\eta))$, en particulier, pour tout ouvert $U$, on a $\dim U=\dim T$.
\eprop
\bpf
Quitte à recouvrir $T$, on peut supposer $T=\Spec(A)$ affine. 
Soit $U=D(f)$ un ouvert affine non-vide de $T$ qui est un $k$-schéma de type fini, on a donc $\degtr_{k}(k(\eta))=\dim(U)\leq \dim(T)$. Montrons la réciproque; comme $T$ est intègre, $U$ est dense. On écrit  $A\simeq\colim A_{\al}$ comme réunion de ses sous-$k$-algèbres de type fini, qui sont aussi intègres et $T\simeq\varprojlim T_\al$. Il existe $\al$ tel que $f\in A_{\al}$, de telle sorte que $A_{f}\simeq\colim\limits_{\beta\geq\al} A_{\beta,f}$.
Comme $A_f$ est une $k$-algèbre de type fini, quitte à grandir $\al$, on peut supposer que $A_{f}\cong A_{\beta,f}$ pour tout $\beta\geq\al$.
Comme les algèbres considérés sont intègres, on obtient la chaîne d'égalités pour $\beta\geq\al$:
\[\dim (U)=\dim (U_{\beta})=\dim (T_{\beta})=\sup\limits_{\beta\geq\al}\dim (T_\beta).\]
Il nous suffit de montrer que $\dim(T)\leq \sup\limits_{\beta\geq\al}\dim (T_\beta)$.
Soit alors une chaîne d'idéaux premiers $\kp_0\subsetneq\dots\subsetneq\kp_{n}$ de longueur maximale de $T$, on considère pour tout $i$, $s_{i}\in\kq_i\backslash\kq_{i-1}$, soit $\beta'$ tel que $s_0,\dots,s_n\in A_{\beta'}$. Pour tout $\beta\geq\beta'$, on a donc par intersection, une chaîne d'idéaux premiers:
\[\kp_0\cap A_{\beta}\subsetneq\dots\subsetneq\kp_{n}\cap A_{\beta}\]
et donc on en déduit que $\dim(T)\leq \sup\limits_{\beta\geq\al}\dim (T_\beta)$, ce qu'on voulait.
\epf

\bcor\label{jaff}
Soit $T$ de type $(S)$, alors c'est un schéma de Jaffard, i.e. pour tout $n\in\NN$, on a $\dim (T\times_k\ab^n)=\dim (T)+n$.
\ecor
\brem Etant donné que l'on considère des schémas non-noethériens, une telle propriété n'est pas automatique, il existe des anneaux $A$ de dimension un pour lesquels $\dim(A[t])=3$.
\erem
\bpf
Comme $T\times_{k}\ab^n$ est de type $(S)$, par récurrence on peut supposer $n=1$ et comme $T$ est noethérien, on peut supposer $T$ intègre et donc $T\times\ab^1$ aussi.
Soit $U$ un ouvert non-vide de $T$ qui est un $k$-schéma de type fini.
D'après la proposition précédente et comme $U$ est de type fini, on a $\dim(T)=\dim(U)$ et $\dim(T\times_{k}\ab^n)=\dim(U\times_{k}\ab^n)=n+\dim(U)=n+\dim(T)$, ce qu'on voulait.
\epf

\section{Construction d'un atlas}
\subsection{Énoncés principaux et dévissages}
On considère un $k$-schéma de type fini $X$ réduit, pur avec $\dim X\geq 1$. 
On dispose de l'ouvert $\clbul:=\cll-\cll_{sing}$, comme c'est un schéma qui n'est pas quasi-compact, on va considérer des ouverts plus petits.
Pour tout entier $d\in\mathbb{N}$, on peut alors considérer l'ouvert 
\[\clx:=\ev_{d}^{-1}(\clxD-\clxD_{sing})
\]
avec $\ev_{d}:\cll\rightarrow\clxD$. C'est un schéma quasi-compact quasi-séparé, comme image inverse d'un schéma de type fini par un morphisme affine.
\begin{theorem}\label{fonda}
Soit $d\in\NN$, alors  il existe un morphisme de schémas $f:Z\rightarrow\clx$ avec $Z$ tel que:
\benumr
	\item 
	$f$ est Weierstrass, affine et surjectif.
	\item
	On a une immersion décente $Z\hookrightarrow T\times\ab^{\NN}$ avec $T$ un $k$-schéma de type (S).
\eenum
Si de plus, $X$ est irréductible, $T$ l'est aussi. On appelle un tel $Z$ un atlas formel.
\end{theorem}
\brem
 On rappelle que les immersions décentes sont des isomorphismes formels.
\erem
On commence par l'énoncé suivant:
\begin{prop}\label{findec}
Soit $S'\rightarrow\clx$ de présentation finie, alors si $\clx$ admet un atlas formel, il en induit un par changement de base sur $S'$.	
\end{prop}
\begin{proof}
Soit $Z\rightarrow \clx$ un atlas formel comme dans le \Cref{fonda}. 
Soit $Z_{1}$  le changement de base de $Z$ à $S'$. Par changement de base, $Z_{1}$ est affine, surjectif et  Weierstrass sur $S'$.
De plus, $Z_1$ est de présentation finie sur $Z$.
On a une immersion décente:
\[\iota:Z\hookrightarrow T\times\ab^{\NN}
\]
avec $T$ de type $(S)$. Comme $T\times\ab^{\NN}$ est quasi-compact quasi-séparé, d'après le \Cref{limnil} et \cite[Tag. 01ZL]{Pro}, il existe une immersion décente de présentation finie $\iota_{\al}:Z_{\al}\rightarrow T\times\ab^{\NN}$ telle que l'on a un carré cartésien:
$$\xymatrix{Z_{1}\ar[d]\ar[r]&Z_{1,\al}\ar[d]\\Z\ar[r]& Z_{\al}}$$
où les flèches horizontales sont des immersions décentes.
En particulier, $Z_{1,\al}$ est de présentation finie sur $T\times\ab^{\NN}$, donc à nouveau par descente noethérienne, on a:
\[Z_{1,\al}\simeq T'\times\ab^{\NN},
\]
où $T'$ est de présentation finie sur $T$, donc il est de type $(S)$, ce qui conclut.
\end{proof}
On dit qu'un $k$-schéma est un schéma \textsl{d'intersection complète spécial}, s'il est défini comme un sous-schéma fermé de $\Spec(k[x_{1},\dots,x_{m},y_{1},\dots, y_{l}])$ par les équations $p_{1}=\dots=p_{l}=0$.
Dans ce cas, le lieu singulier consiste en le lieu d'annulation du déterminant de la matrice jacobienne $(\frac{\partial p_{i}}{\partial y_{j}})$.

\begin{prop}\label{redu}
Le \Cref{fonda} pour les intersections complètes spéciales implique le cas général.
\end{prop}

On a besoin de la proposition suivante qui résulte de \cite{DL} et \cite[Prop.4.1, Lem. 4.2]{ME} :
\begin{prop}\label{dl}
Supposons de plus $X$ affine.
Soient $r, d$ des entiers avec $r\geq d$, alors il existe des schémas intersections complètes spéciaux $M_{i}$ et un recouvrement par des ouverts affines $U_{i}$ de $\clrx^{\leq d}$ qui vérifient les conditions suivantes:
\benumr
\item
On a des immersions fermées $X\hookrightarrow M_{i}$.
\item
Les ouverts $U_{i}$ sont contenus dans $\clrmi^{\leq d}$.
\item
Soit $V_{i}$ l'image réciproque de $U_{i}$ dans $\cll$, alors on a un carré cartésien:
$$\xymatrix{V_{i}\ar[d]\ar[r]&\clmi^{\leq d}\ar[d]\\U_{i}\ar[r]&\clrmi^{\leq d}}$$
\eenum
\end{prop}
Passons maintenant à la preuve de la \Cref{redu}:
\begin{proof}
Tout d'abord, on peut se ramener immédiatement au cas où $X$ est affine en prenant un recouvrement par des ouverts affines.
On considère alors des ouverts $U_{i}$, $V_{i}$ et un schéma intersection complète $M_{i}$, tels que dans la \Cref{dl}. Il suffit de montrer l'existence d'un atlas formel pour $V_{i}$. On a alors un carré cartésien:
$$\xymatrix{V_{i}\ar[d]\ar[r]&\clmi^{\leq d}\ar[d]\\U_{i}\ar[r]&\clrmi^{\leq d}}$$
En particulier, par changement de base, on a une immersion de présentation finie $V_{i}\rightarrow\clmi^{\leq d}$. On conclut alors par la \Cref{findec}.
\end{proof}

En plus de l'atlas, on a besoin d'une bonne stratification de l'espace d'arcs.
\begin{theorem}\label{fonda2}
Le schéma $\clx$ admet une stratification finie constructible par des schémas isomorphes à $H\times\ab^{\NN}$ pour un $k$-schéma de type fini $H$.
\end{theorem}
\begin{prop}
Le \Cref{fonda2} se déduit du cas intersection complète spécial.
\end{prop}
\begin{proof}
On utilise à  nouveau la \Cref{dl} et l'assertion se réduit à montrer que si l'on a une immersion localement fermée de présentation finie:
\[F\rightarrow H\times\ab^{\NN} 
\]
alors $F$ est également de la forme $H_{1}\times\ab^{\NN}$ avec $H_{1}$ un $k$-schéma de type fini et cela se déduit immédiatement par descente noethérienne.
\end{proof}

On montre dans la section \ref{lcifonda}  les théorèmes \ref{fonda}, \ref{fonda2} dans le cas d'intersection complète spécial. 
\subsection{Le cas d'intersection complète spécial}
Soit $A$ un anneau commutatif, une série $x\in A[[t]]$ est \textsl{non-dégénérée} si pour tout morphisme $\nu:A\rightarrow K$ avec $K$ un corps, $\nu(x)\in K[[t]]\cap K((t))^{\times}$.
\begin{lem}\label{dec1}
On a les assertions suivantes:
\benumr
\item
Soit $B\in\Inf_{k}$, alors si $x\in B[[t]]$ est non-dégénérée, on a
$x\in B[[t]]\cap B((t))^{\times}$.
\item
Soit $A$ un anneau décent, une série non-dégénérée $x\in A[[t]]$, alors la multiplication par $x$ est injective.
\eenum
\end{lem}
\begin{proof}
(i) Soit $\km$ l'idéal maximal de $B$.
Comme $\bar{x}=t^{b}u(t)\in k[[t]]\cap k((t))^{\times}$ et que par hypothèse $\km^{r}=(0)$ pour un certain $r\in\NN$, on obtient qu'il existe $N\in\NN$ tel que $x^{N}=t^{bN}H(t)$ avec $H(t)\in B[[t]]^{\times}$, d'où l'on déduit que $x^{N}\in B[[t]]\cap B((t))^{\times}$ ainsi que $x$.

(ii) On commence par prouver l'énoncé pour les objets de $\Inf_{k}$.
Si $B\in\Inf_{k}$ et $x\in B[[t]]$ est non-dégénérée, d'après (i), $x\in B((t))^{\times}$ donc la multiplication sur $B((t))$ est bijective et comme $B[[t]]\rightarrow B((t))$ est injective, on a le résultat pour $B$.
Si $A$ est décent, la multiplication est injective si et seulement si elle l'est pour tout $B\in\Inf_{k}$, et cela se déduit du cas précédent.
\end{proof}

\begin{prop}\label{dec2}
Soit $A$ un anneau décent,  $\phi\in M_{n}(A[[t]])$ tel que $\psi:=\det(\phi)\in A[[t]]$ est une série non-dégénérée, alors l'application linéaire associée:
\[\phi: A[[t]]^{n}\rightarrow A[[t]]^{n}
\]
 est injective et son image consiste en l'ensemble des $u\in A[[t]]^{n}$ tels que:
\[\phi'u=0~\mod\psi.
\]
où $\phi'$ est la matrice adjointe de $\phi$, telle que $\phi'\phi=\psi.I_{n}$.
\end{prop}
\begin{proof}
De l'égalité de Cramer $\phi\phi'=\phi'\phi=\psi.I_{n}$ et du \Cref{dec1}, on en déduit que $\phi$ et $\phi'$  sont injectives.
De plus, si l'on a deux vecteurs $u,\nu\in A[[t]]^{n}$, de l'injectivité de $\phi$ et $\phi'$, on déduit que $\phi (\nu)=u$ si et seulement si $\phi'(u)=\psi \nu$. En particulier, $u\in\Ima\phi$ si et seulement si $\phi'(u)=0~\mod\psi.$
\end{proof}
Soient deux entiers $m,n\in\NN$ avec $m\geq n$.
On définit le schéma intersection complète $X$, comme le sous-schéma fermé de $\Spec k[x_{1},\dots,x_{m-n},y_{1},\dots,y_{n}]$ défini par les équations $f_{1}=\dots=f_{n}=0$. On note $f:=(f_{1},\dots,f_{n})$ et pour $\g:=(x,y)=(x_{1},\dots,x_{m-n},y_{1},\dots,y_{n})\in\ab^{m}$, on considère la matrice:
\[\phi(\g):=(\frac{\partial f_{i}(x,y)}{\partial y_{j}})_{1\leq i,j\leq n}.
\]
et on note $\psi(\g):=\det(\phi(\g))$.
Soit $M:=\cL\ab^{m-n}\times\cQ_{d}\times\cL\bG_{m}\times A_{d+1}^{n}\times\cL\ab^{n}$, on a une immersion fermée $(i,\psi):X\hookrightarrow \ab^{m}\times\ab^{1}$ qui induit une immersion fermée, notée de la même manière $(i,\psi):\cL X\hookrightarrow\cL\ab^{m}\times\cL\ab^{1}$, on considère alors le carré cartésien:
$$\xymatrix{N_{0}\ar[rr]\ar[d]&&\clx\ar[d]^{(i,\psi)}\\M\ar[rr]^-{(\Id,\beta_{d+1},\al_{d})}&&\cL\ab^{m}\times\cL\ab^{\leq d}}$$
où $\al_{d}$ est donnée par \eqref{aldef}, $\beta_{d+1}$  par \eqref{beta} relativement au polynôme $tq$ et $(i,\psi)$ est une immersion fermée par restriction à l'ouvert $\cL\ab^{m}\times\cL\ab^{\leq d}\subset\cL\ab^{m}\times\cL\ab^1$.
Le schéma $N_{0}$ classifie les uplets $(x,q,u,\bar{y},\xi)\in M$ tels que:
\begin{equation}
f(x,\bar{y}+tq\xi)=0,
\label{eqn}
\end{equation}
\begin{equation}
\psi(x,y)=qu~\text{et}~ y=\bar{y}+tq\xi.
\label{eqn2}
\end{equation}
D'après les  propositions \ref{alw} et\ref{betw} et par changement de base, la flèche $N_{0}\rightarrow\clx$ est Weierstrass, affine et surjective. On a de plus, une immersion fermée $N_0\hookrightarrow M$. On considère alors le sous-schéma fermé $N_{1}\subset M$ qui consiste en les uplets $(x,q,v,\bar{y},\nu)\in M$ tels que:
\begin{equation}
\phi'(\g',f(\g'))+tq^{2}\nu=0,
\label{eqn3}
\end{equation}
\begin{equation}
\psi(\g')=qv.
\label{eqn4}
\end{equation}
Ici, on a posé $\g':=(x,\bar{y})$ et  $\phi'(\g',.)$ est l'endomorphisme adjoint de $\phi(\g,.)$.
\begin{prop}\label{isodec}
On a un isomorphisme canonique $\theta:M\rightarrow M$ de telle sorte que la composée:
\[N_{0}\hookrightarrow M\stackrel{\theta}{\rightarrow} M
\]
se factorise en une immersion fermée $N_0\hookrightarrow N_1$.
De plus, cette immersion induit un isomorphisme pour tout anneau commutatif décent.
En particulier, c'est une immersion décente.
\end{prop}
\begin{proof}
On fait le développement de Taylor de \eqref{eqn} par rapport à la variable $y$, soit:
\begin{equation}
f(x,\bar{y}+tq\xi)=f(x,\bar{y})+tq\phi(x,\bar{y})\xi+t^2q^2H(x,y)=0,
\label{tayl}
\end{equation}
avec $H:=(H_{1},\dots,H_{n})\in A[[t]][x,y]$. En appliquant $\phi'(\g',.)$ avec $\g':=(x,\bar{y})$, on obtient:
\begin{equation}
\phi'(\g,f(\g'))+tq\psi(\g')\xi+t^2q^2\phi'(\g',H(x,y))=0.
\label{tayl2}
\end{equation}
D'autre part, en faisant également un développement de Taylor de \eqref{eqn2} on a :
\begin{equation}
\psi(\g')=q(u+tqB(x,y)).
\label{detp}
\end{equation}
On pose  $v:=u+tqB(x,y)\in\cL\bG_m$. Ainsi, \eqref{tayl2} se récrit:
\begin{equation}
\phi'(\g',f(\g'))+tq^2[v\xi+t\phi'(\g',H(x,y))]=0.
\label{tayl13}
\end{equation}
On considère alors la flèche $\theta:M\rightarrow M$ donnée par:
\[(x,q,u,\bar{y},\xi)\mapsto (x,q,v,\bar{y},v\xi+t\phi'(\g',H(x,y))).
\]
Il résulte alors de \eqref{tayl13} que la composée $N_{0}\hookrightarrow M\stackrel{\theta}{\rightarrow} M$ se factorise par $N_1$.
Montrons que $\theta$ est un isomorphisme et que la flèche $N_0\rightarrow N_1$ est bijective sur les $A$-points pour $A$ décent. On considère $(x,q,v,\bar{y},\nu)\in M(A)$, d'après \ref{suivant}, il existe un unique $\xi$ tel que:
\[\nu=v\xi+t\phi'(\g',H(x,y))~\text{avec}~~ y=\bar{y}+q\xi.
\]
et on peut donc poser $u:=v-tB(x,y)$ avec $B(x,y)$ défini par \eqref{detp}, pour obtenir la fonction réciproque de $\theta$. Si de plus,  $(x,q,v,\bar{y},\nu)\in N_1(A)$, alors ils vérifient \eqref{eqn3}, on obtient un point de $N_{0}(A)$ à l'aide de la proposition \ref{dec2} si $A$ est décent.
\end{proof}
\begin{lem}\label{suivant}
	Soit $h:\ab^{n}\rightarrow \ab^{n}$ une application polynomiale, alors pour tout anneau commutatif $A$ et $\la\in A[[t]]^{\times}$, l'application:
		\[A[[t]]^{n}\rightarrow A[[t]]^n
	\]
	$\nu_{0}\mapsto\la\nu_{0}+th(\nu_{0})$ est bijective.
\end{lem}
\begin{proof}
Soit $x\in A[[t]]^n$, on définit la norme $\vert x\vert=2^{-d}$ où $d$ est le plus petit entier tel que la réduction de $x$ modulo $t^{d}$ soit non-nulle.
Pour cette norme, $A[[t]]^n$ est un espace métrique complet, comme on considère la topologie limite projective.

Soit $\nu_{1}\in A[[t]]^n$, il suffit de montrer que l'application 
$\nu\mapsto\phi(\nu)=\la^{-1}[\nu_{1}-th(\nu)]$ admet un point fixe. Comme $h$ est polynomiale, $\vert h(\nu)-h(w)\vert\leq\vert\nu- w\vert$ et $\vert\la\vert=1$ comme $\la$ inversible, on déduit: 
\[\vert \phi(\nu)-\phi(w)\vert\leq\frac{1}{2}\vert\nu- w\vert.
\]
 On obtient donc que $\phi$ est contractante et admet donc un unique point fixe par le théorème de Banach.
\end{proof}
\subsection{Décomposition en morceaux}\label{lcifonda}
On va maintenant utiliser le schéma $N_{1}$ pour séparer la partie de dimension finie de la partie espace affine de dimension infinie.
Dans la suite, pour simplifier les notations, on pose $c(x,\bar{y}):=\phi'(\g',f(x,\bar{y}))$.
On forme le carré cartésien:
$$\xymatrix{N_{2}\ar[d]\ar[rr]^{f_{1}}&&N_{1}\ar[d]^{\pr}\\\cQ_{d}\times A_{d+1}^{m-n}\times\cL\ab^{m-n}\ar[rr]^-{\beta_{d}}&&\cQ_{d}\times\cL\ab^{m-n}}$$
avec 
\begin{equation}
\pr(x,q,v,\bar{y},\nu)=(q,x)
\label{gammadef}
\end{equation}
 et $\beta_{d}$ donnée par:
\[\beta_{d}(q,\bar{x},\xi)=(q,\bar{x}+tq\xi).
\]
A nouveau, comme $\beta_{d}$ est Weierstrass, affine, surjectif, $f_{1}$ l'est aussi. Montrons que $N_{2}$ est isomorphe à un produit d'un schéma de type (S) avec un $\ab^{\NN}$.
Le schéma $N_{2}$ est donné par les équations:
\begin{equation}
c(\bar{x}+tq\xi,\bar{y})+tq^2\nu=0~ \text{avec} ~x=\bar{x}+tq\xi,
\label{eq9}
\end{equation}
\begin{equation}
\psi(x,\bar{y})=qv.
\label{eq10}
\end{equation}
Pour décomposer le schéma $N_2$, on va faire \og une division euclidienne\fg~ par $tq$.
Soit $Z$ un $k$-schéma affine de type fini. On considère le foncteur $Z_{\cQ_{d}}$ sur la catégorie des $k$-algèbres:
\[Z_{\cQ_{d}}(R):=\{(q,x), q\in\cQ_{d}(R), x:\Spec(R[t]/q)\rightarrow Z\}
\]
Il est représentable par un schéma affine de type fini au-dessus de $\cQ_{d}$. Sa fibre au-dessus d'un polynôme séparable est le produit $Z^{d}$ et au-dessus de $t^{d}$, on obtient $\cL_{d-1}Z$.
Dans le cas $Z=\ab^{1}$, on a:
\[\ab^{1}_{\cQ_{d}}(R):=\{(q,x)\vert~q\in\cQ_{d}(R), x\in R[t]/(q)\},
\]
qui est un fibré vectoriel de rang $d$ sur $\cQ_{d}$.
Comme la flèche canonique $A_{d}(R)\rightarrow R[t]/(q)$ est un isomorphisme par la division euclidienne, ce fibré se trivialise en :
\begin{equation}
\cQ_{d}\times A_{d}\cong\ab^{1}_{\cQ_{d}}.
\label{diveucl}
\end{equation}
A tout morphisme $f:Z\rightarrow V$ et tout entier $r\in\NN$, on a par fonctorialité une flèche
\begin{equation}
f_{\cQ_{r}}:Z_{\cQ_{r}}\rightarrow V_{\cQ_{r}}.
\label{foncteq}
\end{equation}
En prenant $r=d+1$ et $r=2d+1$, les morphismes $\psi:\ab^m\rightarrow\ab^{1}$ et $c:\ab^{m}\rightarrow\ab^{n}$ induisent des morphismes:
\[\psi_{\cQ_{d+1}}:\ab_{\cQ_{d+1}}^m\rightarrow\ab_{\cQ_{d+1}}^{1} \text{et}~~ c_{\cQ_{2d+1}}:\ab_{\cQ_{2d+1}}^m\rightarrow\ab_{\cQ_{2d+1}}^{n}.
\]
De plus, on rappelle que dans le cas d'un espace affine, on a par division euclidienne (cf.\eqref{diveucl}) 
\[\forall~k\in\NN, \ab^{k}_{\cQ_{d+1}}\cong\cQ_{d+1}\times A_{d+1}^{k}.
\]
On considère alors la composée:
\[\bar{\psi}:\cQ_{d}\times A_{d+1}^{m}\rightarrow\ab_{\cQ_{d+1}}^m\stackrel{\psi_{\cQ_{d+1}}}{\rightarrow}\cQ_{d+1}\times A_{d+1},
\]
où la première flèche est donnée $(t.,\Id)$ où la première coordonnée désigne la multiplication par $t$ ainsi que
\[\bar{c}:\cQ_{d}\times A_{2d+1}^{m}\rightarrow\ab_{\cQ_{2d+1}}^m\stackrel{c_{\cQ_{2d+1}}}{\rightarrow}\cQ_{2d+1}\times A_{2d+1}^{n}.
\]
où la première flèche est induite par $\cQ_{d}\rightarrow\cQ_{2d+1}$ qui envoie $q$ sur $tq^2$.
Ceci nous permet donc d'écrire:
\begin{equation}
c(\bar{x}+tq\xi,\bar{y})=\bar{c}(\bar{\g},q)+tq^2c'(x,\bar{y}),
\label{eq11}
\end{equation}
\begin{equation}
\psi(x,\bar{y})=\bar{\psi}(\bar{\g},q)+tq\psi'(x,\bar{y}).
\label{eq12}
\end{equation}
avec $\bar{\g}:=(\bar{x},\bar{y})$ et $x=\bar{x}+tq\xi$. On peut donc récrire $N_2$ comme l'espace des uplets $(q,\bar{x},\bar{y},\xi,\nu,v)$ tels que:
\begin{equation}
\bar{c}(\bar{\g},q)+tq^2[c'(x,\bar{y})+\nu]=0,
\label{eq13}
\end{equation}
\begin{equation}
(\bar{\psi}(\bar{\g},q)-qv_0)+tq[\psi'(x,\bar{y})-v']=0.
\label{eq14}
\end{equation}
avec $v=v_0+tv'$.
On a alors une flèche canonique de présentation finie:
\[
\psi:\bG_{m}\times\cQ_{d}\times A_{d+1}^{m}\rightarrow\cQ_{d+1}\times\cQ_{2d+1}\times A^{n}_{2d+1}\times A_{d+1}
\]
donnée par $(v_0,q,\bar{x},\bar{y})\mapsto (tq,tq^2,\bar{c}(\bar{\g},q),(\bar{\psi}(\bar{\g},q)-qv_0))$.
On forme alors le carré cartésien:
$$\xymatrix{T\ar[d]\ar[r]&S_{d+1}\times_{k} S_{2d+1}^{(n)}\ar[d]\\\bG_{m}\times\cQ_{d}\times A_{d+1}^{m}\times\cL\ab^n\times\cL\ab^1\ar[r]^-{\psi}&\cQ_{d+1}\times\cQ_{2d+1}\times A^{n}_{2d+1}\times A_{d+1}\times\cL\ab^n\times\cL\ab^1}$$
où la flèche verticale de droite vient de la définition \eqref{defSd} en tant que produit fibré.
Par changement de base, comme $\psi$ est de présentation finie, $T$ est de présentation finie sur $S_{d+1}\times_{k} S^{(n)}_{2d+1}$. En particulier, il est donc de type $(S)$ d'après \ref{ouv}.
Par définition, il consiste en l'espaces des uplets $(v_0,q,\bar{x},\bar{y},\xi_1,\xi_2)$ tels que:
\begin{equation}
\bar{c}(\bar{\g},q)+tq^2\xi_1=0,
\label{eq15}
\end{equation}
\begin{equation}
(\bar{\psi}(\bar{\g},q)-qv_0)+tq\xi_2=0.=,
\label{eq16}
\end{equation}
avec $\bar{\g}=(\bar{x},\bar{y})$.
Il ne reste plus qu'à montrer le lemme suivant:
\begin{lem}\label{isoN}
La flèche
\[N_{2}\rightarrow T\times\cL\ab^{m-n}.
\]
donnée par :
\[(q,\bar{x},\bar{y},\xi,\nu,v)\mapsto ((v_0,q,\bar{x},\bar{y},c'(x,\bar{y})-\nu,\psi'(x,\bar{y})-v'),\xi),
\]
obtenue à partir de \eqref{eq13} et \eqref{eq14} est un isomorphisme.
\end{lem}
On rappelle que l'on écrit $v=v_0+tv'$, $x=\bar{x}+tq\xi$ et que $N_2$ est déterminé par les équations \eqref{eq13} et \eqref{eq14}.
\begin{proof} 
On construit alors une flèche dans le sens opposé :
\[
T\times\cL\ab^{m-n}\rightarrow N_{2},
\]
donnée par:
\[(v_0,q,\bar{x},\bar{y},\xi_1,\xi_2,\xi)\mapsto ((q,\bar{x},\bar{y},\xi,\nu,v)),
\]
où $\nu:=c'(\bar{x}+tq\xi,\bar{y})-\xi_{1}$, $v':=\psi'(\bar{x}+tq\xi,\bar{y})-\xi_2$ et $v=v_0+tv'$.
On vérifie immédiatement que la composée des deux est un isomorphisme.
\end{proof}

On peut terminer la preuve du \Cref{fonda2}:
\begin{proof}
En combinant \ref{isodec}  et \ref{isoN}, on a un diagramme commutatif:
$$\xymatrix{&N_{0}\ar[d]\ar[r]&\clx\\T\times\cL\ab^{m-n}\simeq N_{2}\ar[r]&N_{1}}$$
Toutes les flèches horizontales sont des morphismes affines, surjectifs de Weierstrass et la flèche verticale est une immersion décente.
En particulier, si on pose $Z:=N_{2}\times_{N_{1}}N_{0}$, il est Weierstrass, affine et surjectif et admet une immersion décente:
\[Z\hookrightarrow T\times\cL\ab^{m-n}, 
\]
ce qui donne l'atlas voulu.
\medskip

Il ne reste plus qu'à démontrer l'énoncé d'irréductibilité:

si de plus, $X$ est irréductible, d'après \cite[Thm. 3.15]{NS}, $\cL X^{\leq d}$ l'est également. En particulier,  l'ouvert $\cL X_{reg}$ est dense dans $\cL X^{\leq d}$. Soit $Z^{0}\subset Z$ l'ouvert réciproque de $\cL X_{reg}$ dans $Z$. Comme $Z\rightarrow\cL X^{\leq d}$ est fidèlement plat, par platitude topologique $Z^{0}$ reste dense dans $Z$.
Or, il résulte de \ref{alreg} et \ref{betreg} qu'au-dessus de $\cL X_{reg}=\cL X^{0}$, les flèches $\al_{d}$ et $\beta_{d}$ sont surjectives, lisses et leurs fibres sont des schémas lisses connexes. On en déduit donc que $Z^{0}$ est également irréductible et donc $Z$ par densité.
Enfin, $Z$ et $T\times\ab^{\NN}$ ont même espace topologique donc $T\times\ab^{\NN}$ est irréductible et $T$ également.
\end{proof}
Il reste à montrer le \Cref{fonda2}, il suffit de montrer la proposition suivante:
\begin{prop}
Pour tout $d$, le schéma $\clxd$ est isomorphe à $D_{d}\times\cL\ab^{m}$ avec $D_{d}$ qui est un $k$-schéma de type fini.
\end{prop}
\begin{proof}
On considère le foncteur $B_{d}$ sur la catégorie des $k$-algèbres donné par:
\[R\mapsto B_{d}(R):=\{\bar{z}\in R[[t]]^{m}, f(\bar{z})=0~[t^{2d+1}],\psi(\bar{z})=t^{d}u, u\in R[[t]]^{\times}\}.
\]
Comme les équations ne portent que sur la réduction de $\bar{z}$ modulo $t^{2d+1}$, le foncteur est représentable par un schéma de la forme 
$D_{d}\times\cL\ab^{m}$ avec:
\[D_{d}(R):=\{\bar{z}\in (R[t]/(t^{2d+1}))^{m}, f(\bar{z})=0~[t^{2d+1}],\psi(\bar{z})=t^{d}u, u\in R[[t]]^{\times}\}.
\]
Pour toute $k$-algèbre $R$, si $z\in\cL X^{=d}(R)$, on commence par l'écrire $z=\bar{z}+t^{2d+1}\nu\in R[[t]]^{m}$ avec $\nu:=(0,\dots,0,\nu_{1},\dots,\nu_{n})\in R[[t]]^{n}$. En faisant un développement de Taylor de $f(\bar{z}+t^{2n+1}\nu)$, on obtient immédiatement que $\bar{x}$ définit un point de $B_{d}(R)$.
Pour démontrer l'énoncé, il nous suffit de vérifier que les foncteurs sont isomorphes pour tout anneau commutatif $A$.
Si on a un point $\bar{x}\in B_{d}(A)$, le même argument que pour la \Cref{dec2} nous fournit l'existence et unicité d'un $\nu$ tel que:
\[f(x+t^{2d+1}\nu)=0.
\]
Il est à noter que dans ce cas on n'a pas besoin de la condition de décence étant donné que $\psi(\bar{x})=t^{d}u, u\in A[[t]]^{\times}$ n'est jamais un diviseur de zéro. On obtient donc le résultat souhaité.
\end{proof}
\brem Il est indispensable de travailler avec des schémas de type (S).
En effet, ces schémas non-noethériens permettent de recoller les différents modèles formels qui apparaissent dans le théorème de Drinfeld-Grinberg-Kazhdan. En revanche, cela ne peut être fait avec des $k$-schémas de type fini, ainsi que le montre l'exemple suivant.
\erem
\medskip

\textbf{Exemple}: On considère le cas du cône quadratique $X:=\{(x,y,z)\vert xy=z^2\}$.

Si l'on prend les arcs de valuation un et deux, $\g_{1}(t)=(t,0,0)$ et $\g_{2}(t)=(t^2,0,0)$, les modèles formels $Y_{1}$ et $Y_{2}$ sont donnés par $\{v^2=0\}$ et $\{(a,b,v,w)\vert~ aw^{2}=v^{2}, bw^{2}=2wv\}$.
On a alors un morphisme:
\[\phi:Y_{2}\times\cL\bG_{m}\times\cL\ab^{1}\ra\cL X^{\leq 2},
\]
donné par $((a,b,v,w),u(t),\xi(t))\mapsto (qu(t),y(t),v+wt+q\xi(t))$, avec $y(t)$ défini uniquement par $(x(t),z(t))$ et $q=a+bt+t^2$.
Le morphisme $\phi$ induit un isomorphisme formel sur les points qui s'envoient sur $\cL X^{=2}$. 
En revanche, cela n'est plus vrai aux autres points où les singularités des deux côtés sont différentes.
En effet, le quadruplet $(1,0,0,0)$ est un point singulier de $Y_2$ mais s'envoie sur un point lisse de $\cL X$.

\section{Cohomologie étale des espaces d'arcs}
\subsection{Constructibilité pour les $S$-schémas}
Soit $k$ un corps, on considère $n\in\NN^{*}$ premier à la caractéristique. Dans la suite, tous les foncteurs sont dérivés. L'énoncé principal est \ref{immc}.
On commence par rappeler le théorème de Gabber \cite[Exp. XX, 4.4]{ILO} qui est un énoncé de comparaison avec la complétion:
\begin{theorem}\label{base}
Soit $A$ un anneau commutatif, $I\subset A$ un idéal de type fini.
On considère un morphisme $(A,I)\rightarrow (A',I')$ de paires henséliennes avec $I'=IA'$, de telle sorte que l'on a un isomorphisme au niveau des complétions $I$ et $I'$-adiques:
\[A^{\wedge}\simeq A'^{\wedge}.
\]
Soient $X=\Spec(A)$, $X'=\Spec(A')$, $U:=\Spec(A)-V(I)$, $U':=\Spec(A')-V(I')$, $\pi:X'\rightarrow X$, $j:U\rightarrow X$, $j':U'\rightarrow X'$. On suppose $n$ inversible sur $X$, alors pour tout  $K\in D^{b}_{c}(U,\bZ/n)$, le morphisme de changement de base:
\[\pi^{*}j_{*}K\rightarrow j'_{*}\pi^{*}K.
\]
est un isomorphisme.
\end{theorem}
\brems\label{deviss}
\remi 
Il est important de noter que l'on ne fait pas d'hypothèses sur $A$.
\remi
Dans loc.cit., l'énoncé est au niveau des faisceaux de cohomologie et pour $K$ un faisceau de $\bZ/n$-modules. Il s'étend dans notre cas de la manière suivante: d'après \cite[Tag. 095S]{Pro}, la flèche de changement de base est définie dans la catégorie dérivée; pour montrer que c'est un isomorphisme, il suffit de le voir au niveau des faisceaux de cohomologie.
Par un argument de suite spectrale \cite[Tag. 015J]{Pro}, on se ramène ensuite à $K$ qui est un faisceau de $\bZ/n$-modules et là c'est le théorème \cite[Exp. XX, 4.4]{ILO}.
\erems
On utilisera à plusieurs reprises le corollaire suivant:
\begin{cor}\label{base2}
Soit une paire $(A,I)$ avec $I$ un idéal de type fini, $X:=\Spec(A)$, $U:=X-V(I)$ et $X':=\Spec(A^{\wedge})$ où $A^{\wedge}$ est la complétion $I$-adique de $A$, alors on a un énoncé de changement de base pour $j_{*}$ avec $j:U\rightarrow X$ comme ci-dessus.
\end{cor}
\begin{proof}
En effet, il suffit dans un premier temps de faire le changement de base à l'hensélisé $(A^{h},I^{h})$  de $A$ le long de $I$ et ensuite on applique le \Cref{base}.
\end{proof}
\begin{prop}\label{base1}[Changement de base pro-lisse]
	On considère un carré cartésien de schémas qcqs:
	$$\xymatrix{T_{1}\ar[d]_{\tilde{h}}\ar[r]^{\tilde{g}}&T\ar[d]^{h}\\S_{1}\ar[r]^{g}&S}$$
	avec $g$ pro-lisse et $h$ qcqs. Alors, pour tout $K\in D^{+}(\Tet,\bZ/n)$ le morphisme canonique:
		\[g^{*}h_{*}K\rightarrow \tilde{h}_{*}\tilde{g}^{*}K\]
	est une équivalence.
	\end{prop}
\begin{proof}
Il suffit de montrer l'énoncé en prenant les faisceaux de cohomologie.
Le même argument de suite spectrale  que la remarque \ref{deviss}.(ii) nous ramène au $K$ est un faisceau constructible de $\bZ/n$-modules.
L'énoncé se déduit alors de \cite[Exp. XIV, Lem. 2.5.3]{ILO}.
On remarque que dans loc. cit., la preuve est faite pour un morphisme régulier, mais l'hypothèse clé est en fait la pro-lissité.
\end{proof}
\blem\label{pulcons}
Soit $f:Y\rightarrow S$ un morphisme surjectif de schémas qcqs, alors pour tout $K\in D^b_{c}(S,\bZ/n)$, $K$ est constructible si et seulement si $f^{*}K$ l'est.
\elem
\bpf
Le sens direct est immédiat. Comme $f^{*}$ est exact, il suffit de vérifier l'assertion dans le cas où $K$ est un faisceau constructible de $\bZ/n$-modules et alors cela se déduit de \cite[Tag. 095Q]{Pro}.
\epf
Passons maintenant au théorème clé:
\begin{theorem}\label{immc}
Soit $T$ un schéma de type $(S)$, soit $j:U\rightarrow T$ une immersion ouverte quasi-compacte. Soit $K\in D^{b}_{c}(U,\bZ/n)$, alors $j_{*}K$ est constructible.
\end{theorem}
\begin{proof}
Soit $d\in\NN$, tel que $T$ est de type $(S_d)$.
On procède par récurrence sur $d$.
Si $d=0$, alors $T$ est un $k$-schéma de type fini et c'est l'énoncé de constructibilité usuelle \cite[Exp. XIX, 5.1]{SGA4}.
Supposons la propriété vraie pour tout $k\leq d$ et montrons-là pour $d+1$. On écrit alors $T= V_{d-1}\sqcup F_{d-1}$ avec $V_{d-1}$ de type $(S_{d-1})$ et $F_{d-1}$ un $k$-schéma de type fini.
On a sur $U$ un triangle distingué:
\[i_{*}i^{!}K\rightarrow K\rightarrow h_{*}h^{*}K
\]
avec $h:U_{d-1}:=U\cap V_{d-1}\rightarrow U$ et $i:H_{d-1}:=U-V_{d-1}\rightarrow U$ l'immersion fermée complémentaire.
Comme $K$ est constructible, $h^{*}K$ l'est aussi. Montrons que $h_*h^{*}K$ est constructible  sur $U$. Le problème étant local, on peut donc supposer que $U$ est affine.
D'après \ref{dim-comp} le schéma affine $U^{\wedge}$ sous-jacent à la complétion formelle de $U$ le long de $H_{d-1}$ est noethérien excellent. On considère alors le diagramme cartésien:
$$\xymatrix{U_{d-1}^{\wedge}\ar[d]\ar[r]&U^{\wedge}\ar[d]\\U_{d-1}\ar[r]&U}$$
avec $U_{d-1}^{\wedge}$ l'image inverse dans $U^{\wedge}$ de $U_{d-1}$.
D'après \ref{pulcons} et \ref{base2}, il suffit de montrer la constructibilité au niveau de la flèche du haut, qui est une immersion ouverte entre schémas noethériens excellents. Cela se déduit alors du théorème de constructibilité de Gabber \cite[Intro, Thm.1]{ILO}.
Ainsi, $h_*h^{*}K$ est constructible et comme $K$ l'est aussi, on en déduit la constructibilité de $i_*i^!K$.
En appliquant $j_{*}$ à ce triangle, on obtient un triangle distingué:
\[j_{*}(i_*i^!K)\rightarrow j_{*}K\rightarrow j_{*}(h_{*}h^{*}K).
\]
Il suffit donc de montrer que les deux extrémités sont constructibles.
 Pour l'extrémité de droite, on a un diagramme commutatif:
$$\xymatrix{U_{d-1}\ar[r]^{j'}\ar[d]_{h}&V_{d-1}\ar[d]^{h'}\\U\ar[r]^{j}&T}$$
On en déduit:
\[j_{*}(h_{*}h^{*}K)=h'_{*}j'_{*}(h^*K).
\]
Comme $V_{d-1}$ est de type $(S_{d-1})$, par hypothèse de récurrence $j'_{*}(h^{*}K)$ est constructible et également $h'_{*}j'_{*}(h^*K)$ en appliquant à nouveau \ref{pulcons}, \ref{base2} et \cite[Intro, Thm.1]{ILO}.
Enfin la constructibilité de $j_{*}i_*i^!K$ se déduit du carré commutatif:
$$\xymatrix{H_{d-1}\ar[r]\ar[d]_{i}&F_{d-1}\ar[d]\\U\ar[r]^{j}&T},$$
 du fait que $F_{d-1}\rightarrow T$ est une immersion fermée et que $H_{d-1}\rightarrow F_{d-1}$ est une immersion ouverte entre $k$-schémas de type fini donc préserve les constructibles.
\end{proof}

\begin{cor}\label{fin1}
Soit $f:T'\rightarrow T$ un morphisme  séparé de type fini entre schémas de type $(S)$, alors les foncteurs $(f_{*},f_{!},f^{*},f^{!})$ préservent les constructibles.
\end{cor}
\brem On rappelle que dans \cite[Exp XVIII, Thm.3.1.4]{SGA4}, Deligne définit le foncteur $f^!$ pour tout morphisme compactifiable, i.e. qui se factorise en une immersion ouverte suivie d'un morphisme propre, ce qui est le cas d'un morphisme de séparé de type fini d'après Nagata \cite[Tag. 0F41]{Pro}.
\erem
\begin{proof}
Pour $f^{*}$ c'est immédiat (\cite[Tag. 095G]{Pro}), pour $f_{!}$, cela se déduit du changement de base propre \cite[Exp. XIV]{SGA4}.
Pour $f_{*}$, en utilisant la compactification de Nagata, $f$ se factorise en $f=g\circ j$ avec $g$ propre et $j$ une immersion ouverte quasi-compacte.
En particulier, le résultat se déduit du \Cref{immc} et de la finitude de $g_{*}$ comme $g$ est propre.
Il reste donc le cas de $f^{!}$. Comme l'assertion est locale sur $T'$, on factorise $f$ en $f=p\circ i$ avec $i$ une immersion fermée de type fini et $p$ lisse.
Comme $p$ est lisse, on $p^{!}=p^{*}[2d]$ où $d$ est la fonction localement constante dimension relative de $p$ (d'après \cite[Exp XVIII,Thm. 3.2.5]{SGA4}). Il suffit donc de montrer le résultat pour $f=i$.
Soit $\la$ l'immersion ouverte complémentaire, qui est quasi-compacte, $K\in D_{c}^{b}(T,\bZ/n)$, on a un triangle distingué:
\[i_{*}i^{!}K\rightarrow K\rightarrow \la_{*}\la^{*}K
\]
D'après \ref{immc}, comme $K$ est constructible, $\la_{*}\la^{*}K$ l'est aussi et par conséquent $i_{*}i^{!}K$.
Comme $i^{!}K=i^{*}(i_{*}i^{!}K)$, on conclut.
\end{proof}
\begin{cor}\label{fin2}
Soit $T$ de type $(S)$, $K,L\in D^{b}_{c}(T,\bZ/n)$, alors $\Rhom(K,L)\in D^{b}_{c}(T,\bZ/n)$.
\end{cor}
\begin{proof}
L'assertion est locale sur $T$ et en filtrant $K$, d'après \cite[Exp. IX, Prop.2.5]{SGA4}, on peut supposer que $K\simeq k_{!}\bZ/n$ où $k$ est une immersion localement fermée constructible.
On a alors:
\[\Rhom(K,L)=\Rhom(k_{!}\bZ/n,L)\simeq k_{*}k^{!}L,
\]
et l'assertion se déduit de \ref{fin1}.
\end{proof}

\subsection{Constructibilité sur les  $\cL$-schémas}
 Dans cette section, le but est d'obtenir l'énoncé de finitude \ref{finitude0}.

Soit $S$ un $k$-schéma qcqs, on dit que c'est un \textsl{$\cL$-schéma} s'il admet un atlas formel tel que dans le théorème \ref{fonda} et une stratification finie constructible du type du théorème \ref{fonda2}. Il résulte de la preuve de la proposition \ref{findec} que tout schéma de présentation finie sur un $\cL$-schéma est un $\cL$-schéma.
On commence par un cas particulier de \ref{finitude0}:
	\begin{prop}\label{partic}
	Soit $f:Z_2\rightarrow Z_1$ un morphisme séparé de présentation finie de $\cL$-schémas avec $Z_1=T\times\ab^{\NN}$ où $T$ est de type (S), alors les foncteurs $(f_{!}, f_{*}, f^{!}, f^{*})$ préservent les catégories constructibles.
	\end{prop}
	\begin{proof}
	Tout d'abord, il est à noter que comme $f$ est séparé de présentation finie, il se descend à un cran fini de telle sorte que $Z_2$ provient d'un schéma $\overline{Z}_{2}$ séparé de présentation finie sur un $T\times\ab^{r}$ avec
	\[Z_{2}\simeq \overline{Z}_{2}\times_{T\times\ab^{r}}Z_{1}.\]
	et $\overline{Z}_{2}$ est donc aussi de type (S).
	On commence par la paire $(f_!,f_*)$, soit $K\in D^{b}_{c}(Z_2,\bZ/n\bZ)$.
Par un argument de suite spectrale \cite[Tag. 015J]{Pro}, on se ramène à $K$ qui est un faisceau constructible de $\bZ/n$-modules.
On utilise alors \cite[Tag. 09YU]{Pro} pour dire que $K$ se descend à un cran fini et sans restreindre la généralité, quitte à agrandir $\overline{Z}_2$, on peut supposer que $K=h^{*}K_0$ avec $K_{0}\in\Sh_{c}(\overline{Z}_{2,\et})$ et $h:Z_2\rightarrow\overline{Z}_2$ est pro-lisse puisque qu'il s'obtient par changement de base de la projection:
\[T\times\ab^{\NN}\rightarrow T\times\ab^{r}.\]
Pour $f_!$, l'énoncé se déduit du changement de base propre et du corollaire \ref{fin1}, pour $f_*$ cela se déduit des propositions \ref{base1} et \ref{fin1}.
Passons à la paire $(f^*, f^!)$, pour $f^*$ c'est immédiat (\cite[Tag. 095G]{Pro}).
Pour $f^!$ le même argument que \ref{fin1} montre qu'il suffit de traiter le cas d'une immersion fermée $i$ et dans ce cas en utilisant le triangle distingué avec l'immersion ouverte complémentaire $j$, cela se déduit du fait que $(j_*,j^{*}, i^*)$ préservent les constructibles.
	\end{proof}
	
	\begin{theorem}\label{finitude0}
	Soit $f:S_1\rightarrow S$ un morphisme séparé de présentation finie de $\cL$-schémas, alors les foncteurs $(f_{!}, f_{*}, f^{!}, f^{*})$ préservent les catégories constructibles.
	\end{theorem}
	\begin{proof}
	Pour $f^*$, c'est immédiat par construction du foncteur. On commmence par montrer l'énoncé pour $(f_{!}, f_{*})$.
On considère $Z$ un atlas formel de $S$, soit $Z_{1}:=S_{1}\times_{S}Z$.
D'après \ref{base1}, \ref{pulcons} et  par le changement de base propre, il suffit de montrer l'énoncé en remplaçant $S_{1}$ par $Z_{1}$.
Soit $K\in D^b_c(Z_{1,\et},\bZ/n\bZ)$, d'après \ref{fonda}, on a une immersion décente:
\[Z\hookrightarrow T\times\ab^{\NN}\]
que l'on peut écrire comme une limite d'immersions fermées de présentation finie d'après \ref{limnil} où $T$ est de type (S). En particulier, on peut descendre $Z_1$ en $\bar{Z}_1$ séparé de présentation finie sur $T\times\ab^{\NN}$ et quitte à agrandir $\bar{Z}_1$, on peut supposer que $K$ se descend aussi, l'énoncé se déduit alors de \ref{partic} et de l'invariance topologique du topos étale (\cite[Tag. 03SI]{Pro}).
Pour $f^!$, le même argument que \ref{partic} permet de se ramener au cas où $f$ est une immersion fermée $i$ et on conclut de même que pour \ref{partic} à l'aide de la constructiblité de $(j_*,j^*,i^*)$ où $j$ est l'immersion ouverte complémentaire.
	\end{proof}
		
	\begin{prop}\label{rhom}
		Soient $K, L\in D^{b}_{c}(S,\bZ/n)$ avec $S$ un $\cL$-schéma, alors $\Rhom(K,L)\in D^{b}_{c}(\Set,\bZ/n)$.
	\end{prop}
	\begin{proof}
	L'assertion est locale  sur $S$, on peut supposer, quitte à filtrer $K$ qu'il est de la forme $i_{!}\bZ/n$ pour une immersion localement fermée constructible. On a alors par adjonction:
\[\Rhom(K,L)=i_{*}\Rhom(\bZ/n,i^{!}L)=i_{*}i^{!}\bZ/n\]
qui est constructible par le théorème \ref{finitude0}.
	\end{proof}
On peut maintenant en déduire le résultat suivant de finitude de la cohomologie qui étend la situation de Drinfeld \cite[6.8-6.9]{Dr}:
\begin{theorem}\label{finitude}
	Soient $S$ un $\cL$-schéma, $n\in\NN$ inversible sur $S$, et $K\in D^{b}_{c}(S,\bZ/n)$, alors $R\Gamma(S,K)\in D^{b}_{c}(\Spec(k),\bZ/n)$.
	En particulier, pour tout $d\in\NN$, on a $R\Gamma(\cL X^{\leq d},K)\in D^{b}_{c}(\Spec(k),\bZ/n)$.
\end{theorem}
\begin{proof}
On a vu que pour une immersion constructible, les foncteurs $(k_{*},k_{!},k^{*},k^{!})$ préservent les constructibles.
En particulier, en considérant les triangles distingués d'excision et comme $S$ est un $\cL$-schéma, on est ramené au cas produit $H\times\ab^{\NN}$ avec $H$ un $k$-schéma de type fini.
Par un argument de suite spectrale, on peut supposer que $K$ est un faisceau constructible de $\bZ/n$-modules.
D'après \cite[Tag. 09YU]{Pro}, $K$ se  descend à un cran fini en un faisceau constructible $K_{r}$ sur $H\times\ab^{r}$ pour un certain $r$.
D'après \cite[Exp. XVI 1.2]{SGA4}, on a par passage à la limite l'isomorphisme:
\begin{center}
$R\Gamma(H_{\infty},K)=\varinjlim\limits_{m\geq r} R\Gamma_{c}(H_{m},h_{mr}^{*}K_{r})=R\Gamma(H_{r},K_{r})$
\end{center}
avec $h_{mr}:H_{m}\rightarrow H_{r}$ et où la dernière égalité se déduit par acyclicité des fibres. L'énoncé se déduit alors de l'énoncé à un cran fini.
\end{proof}
\subsection{Remarques s'agissant des $t$-structures}\label{details}
Une des motivations originales de ce travail était d'établir une théorie des faisceaux pervers sur des espaces d'arcs. Et plus précisément une théorie qui serait compatible à l'énoncé de Drinfeld-Grinberg-Kazhdan.
Pour tout espace d'arcs $\cL X$, on veut définir un complexe d'intersection $IC_{\cL X}\in D^{b}_{c}(\cL X^{\leq d},\bql)$ tel que pour tout $\g\in(\cL X^{\leq d})(\bar{k})$, on ait modulo des décalages:
\begin{equation}
f_{\g}^{*}IC_{\cL X}\simeq IC_{Y_{\g}^{\wedge}},
\label{pullback}
\end{equation}
où $Y_{\g}^{\wedge}$ est le schéma affine sous-jacent à un modèle formel de dimension finie de $\cL X^{\wedge}_{\g}$ obtenu par DGK et $f_\g:Y_{\g}^{\wedge}\ra\cL X$.
Comme localement pour la topologie pro-lisse les schémas $\cL X^{\leq d}$ se comportent comme des produits $T\times\ab^{\NN}$ pour des schémas de type $(S)$, une première étape est donc d'étudier la situation pour des schémas de type $(S)$.

On considère donc un morphisme affine de présentation finie $f:T\rightarrow S_1$.
On note $F=f^{-1}(\{0\})$ qui est donc un $k$-schéma de type fini, de telle sorte que le schéma affine $T^{\wedge}$ sous-jacent à la complétion de $T$ le long de $F$ est noethérien excellent. On note $U$ son ouvert complémentaire qui est également un $k$-schéma de type fini.

On veut donc comprendre la nature du morphisme de complétion $T^{\wedge}\ra T.$ Elle va être contrôlée par un sous-schéma fermé de $T$.
On a une flèche canonique $S_1\ra\cQ_1$ donnée par $(v,q,\xi)\mapsto q$ qui admet une section $\sigma:\cQ_1\ra S_1$, on forme alors $T_0:=\sigma^{*}T$, c'est un sous-schéma fermé de $T$, qui est de présentation finie sur $\cQ_1$, donc c'est un $k$-schéma de type fini.
On a alors le diagramme commutatif suivant:
$$\xymatrix{T^{\wedge}\ar[d]\ar[rr]&&T\ar[d]\\S_{1}^{\wedge}\cong\Spec(k[[a]])\ar[r]&\cQ_1\ar[r]^{\sigma}&S_1}$$
où $S_1^{\wedge}=\Spec(\cO_{S_1,0}^{\wedge})\cong\Spec(k[[a]])$ d'après la section \ref{sd}.
Ainsi, la flèche de complétion se factorise $T^{\wedge}\ra T$ se factorise via le sous-schéma fermé $T_0$ et on montre de plus que l'on a un isomorphisme:
\[T^{\wedge}\simeq T_0^{\wedge}\] où $T_0^{\wedge}$ est la complétion le long de $F\subset T_0$.
En particulier, comme $T_0$ est un $k$-schéma de type fini, donc excellent, le morphisme $T_0^{\wedge}\ra T_0$ est pro-lisse et donc le morphisme $T^{\wedge}\ra T$ s'identifie à la composée d'un morphisme pro-lisse avec une immersion fermée.

Maintenant, si l'on dispose d'un complexe d'intersection $IC_{T}$, 
on doit avoir $IC_{T}=j_{!*}IC_{U}$ pour $j:U\hra T$.
Or, on a le carré cartésien suivant:
$$\xymatrix{U^{\wedge}\ar[r]\ar[d]&T^{\wedge}\ar[d]\\ U\cap T_0\ar[d]\ar[r]&T_{0}\ar[d]\\U\ar[r]& T}.$$
Ainsi, si l'on veut obtenir une identité $f^{*}IC_{T}=IC_{T^{\wedge}}$ (en négligeant les décalages), il faut déjà que cette égalité soit vraie au-dessus de $U$.
Le problème est que la flèche $U^{\wedge}\ra U$ se factorise par l'immersion fermée $U\cap T_0\ra U$, laquelle n'a aucune raison de préserver le complexe d'intersection de $U$.
On se retrouve donc confronté au fait que les singularités sur $U$ sont différentes de celles de $U^{\wedge}$ et il n'y a donc pas d'espoir d'obtenir un complexe d'intersection compatible à la décomposition de DGK.

\address{ 
  \bigskip
  \footnotesize

  (A.\ Bouthier) \textsc{Sorbonne Université, IMJ-PRG, 4 Place Jussieu, 75005 Paris}\par\nopagebreak
  \textit{E-mail address}: \texttt{alexis.bouthier@imj-prg.fr}
	}
\end{document}